\documentclass[10pt]{amsart}
\usepackage{color}
\usepackage{amssymb}
\usepackage{amsmath}
\usepackage{amsfonts}
\usepackage{latexsym,enumerate}
\usepackage{amsmath,amsthm,amsopn,amstext,amscd,amsfonts,amssymb}
\usepackage[dvips]{graphicx}
\usepackage[latin1]{inputenc}

\theoremstyle{plain}

\newtheorem{remark}{Remark}

\numberwithin{equation}{section}
\long\def\salta#1{\relax}

\newcommand{\dint}{\dyle\int}

\newcommand{\re}{{I\!\!R}}
\newcommand{\ren}{\re^N}

\newcommand{\dyle}{\displaystyle}
\newcommand{\ene}{{I\!\!N}}

\newcommand{\io}{\int\limits_\O}

\renewcommand{\a }{\alpha }
\renewcommand{\b }{\beta }

\newcommand{\D }{\Delta }
\newcommand{\e }{\varepsilon }

\newcommand{\s }{\sigma }

\renewcommand{\O }{\Omega }

\newenvironment{pf}{\noindent{\sc Proof}.\enspace}{\rule{2mm}{2mm}\medskip}

\newtheorem{Theorem}{Theorem}[section]

\newtheorem{Definition}[Theorem]{Definition}
\newtheorem{Lemma}[Theorem]{Lemma}

\newcommand{\cqd}{{\unskip\nobreak\hfil\penalty50
        \hskip2em\hbox{}\nobreak\hfil\mbox{\rule{1ex}{1ex} \qquad}
        \parfillskip=0pt \finalhyphendemerits=0\par\medskip}}
\begin{document}
\title[On  fractional $p$-laplacian parabolic problem]{On  fractional P-laplacian parabolic problem with general data}
\author[B. Abdellaoui, A. Attar, R. Bentifour \& I. Peral]{B. Abdellaoui$^*$, A. Attar$^*$, R. Bentifour$^*$ \& I. Peral$^\dag$}
\thanks{ This work is partially supported by project
MTM2013-40846-P, MINECO, Spain. }
\address{\hbox{\parbox{5.7in}{\medskip\noindent {$*$Laboratoire d'Analyse Nonlin\'eaire et Math\'ematiques
Appliqu\'ees. \hfill \break\indent D\'epartement de
Math\'ematiques, Universit\'e Abou Bakr Belka\"{\i}d, Tlemcen,
\hfill\break\indent Tlemcen 13000, Algeria.\\[3pt]
$\dag$Departamento de Matem{\'a}ticas, U. Autonoma
de Madrid, \hfill\break\indent 28049 Madrid, Spain.\\[3pt]
        \em{E-mail addresses: }{\tt boumediene.abdellaoui@inv.uam.es, \tt ahm.attar@yahoo.fr, \tt rachidbentifour@gmail.com, \tt ireneo.peral@uam.es}.}}}}

%%%%%%%%%%%%%%%%%%%%%%%%%%%%%%%%%%%%%%%%%%%%%%%

\begin{abstract}
In this article the problem to be studied is the following
$$
(P)
\left\{
\begin{array}{rcll}
u_t+(-\D^s_{p}) u & = & f(x,t)  &
\text{ in } \O_{T}\equiv \Omega \times (0,T), \\
u & = & 0 & \text{ in }(\ren\setminus\O) \times (0,T), \\
u & \ge & 0 &
\text{ in }\ren \times (0,T),\\
u(x,0) & = & u_0(x) & \mbox{  in  }\O,
\end{array}%
\right.
$$
where $\Omega$ is a bounded domain, and $(-\D^s_{p})$ is the  fractional p-Laplacian operator defined by
$$ (-\D^s_{p})\, u(x,t):=P.V\int_{\ren} \,\dfrac{|u(x,t)-u(y,t)|^{p-2}(u(x,t)-u(y,t))}{|x-y|^{N+ps}} \,dy$$
with $1<p<N$, $s\in (0,1)$ and $f, u_0$ are measurable functions.

The main goal of this work is to prove that if $(f,u_0)\in L^1(\O_T)\times L^1(\O)$, problem $(P)$ has a weak  solution with suitable regularity. In addition, if $f_0, u_0$ are nonnegative, we show that the problem above has a nonnegative entropy solution.

In the case of nonnegative data, we give also some quantitative and qualitative properties of the solution according the values of $p$.
\end{abstract}

\maketitle

\section{Introduction.}\label{sec:s0}
This work deals with the following parabolic problem
\begin{equation}\label{eq:def}
\left\{
\begin{array}{rcll}
u_t+(-\D^s_{p}) u & = & f(x,t)  &
\text{ in } \O_{T}=\Omega \times (0,T), \\
u&\ge& 0 &\mbox{  in  }\ren\times (0,T)\\
u &= & 0 & \text{ in }(\ren\setminus\O) \times (0,T), \\
u(x,0) &= & u_0(x)& \mbox{  in  }\O,
\end{array}%
\right.
\end{equation}
where $\O$ is a bounded domain, $s\in (0,1), 1<p<N$ and
$$ (-\D^s_{p})\, u(x,t):=\int_{\mathbb{R}^{N}} \,
\frac{|u(x,t)-u(y,t)|^{p-2}(u(x,t)-u(y,t))}{|x-y|^{N+ps}}\ dy$$
is the fractional $p-$laplacian operator which is, in particular,    non local. The data $f$ and $u_0$ are measurable functions under suitable hypotheses that we will precise in each instance.

 For the local  $p-$laplacian operator  there are a large number of references in the literature. Among all of them we refer to \cite{Pri} where the author proved the existence of an entropy solution for all data in $(f,u_0)\in L^1(\O_T)\times L^1(\O)$. The case of general measure data was studied in \cite{BM}, \cite{BMR}, where the existence of renormalized solution is obtained.

Respect to the non local operator, the case $p=2$ has been  analyzed in \cite{LPPS}. Using duality and approximation arguments, the authors proved the existence and the uniqueness of the solution that belongs to a suitable fractional Sobolev space. The case with Hardy potential and under "natural" condition on $(f,u_0)$  has been studied in \cite{AMPP}.

In \cite{VAZ} and \cite{MRT} for $p\neq 2$ and $f\equiv 0$, the authors obtained the existence of energy solution for all $u_0\in L^2(\O)$, explaining the  asymptotical behavior with respect to properties of the corresponding Barenblatt type solution (for $p>2$).

The main goal of this paper is to consider the case $p\neq 2$ with more general data $(f,u_0)\in L^1(\O_T)\times L^1(\O)$. We will prove the existence of a weak \textit{solution obtained as limit of approximations} (SOLA) that belongs to a suitable fractional Sobolev space. Moreover if the data are nonnegative we will prove that a such solution is an entropy solution.

It is worthy to point out that the stationary problem  has been studied in \cite{KMS} and \cite{AAB}. We will use the functional results explained in \cite{AAB} and  some techniques there.

More precisely,  the paper is organized as follow.

In Section \ref{sec0}  we will give some concepts in which the solutions are considered and some functional tools and algebraic inequalities  that will be used  along of  the paper.

Section \ref{sec1} is devoted to  prove the existence of a weak solution for all data $(f,u_0)\in L^1(\O_T)\times L^1(\O)$. The  idea is to proceed by finding a solution as limit of approximations.

Section \ref{s2} is devoted to introduce the concept of entropy solution and to prove that a SOLA is an entropy solution.

In the last section we analyze some qualitative properties of the solutions  related to  the extinction in finite time and the finite speed of  propagation, that is different to the local case.

\section{Preliminaries and functional setting}\label{sec0}

In this section we give some functional settings that will be used
below, we refer to \cite{DPV} and \cite{Adams} for more details.

Let $s\in (0,1)$ and $p>1$, assume that $\O\subset \ren$, the fractional Sobolev
spaces $W^{s,p}(\Omega)$, is defined by
$$
W^{s,p}(\Omega)\equiv
\Big\{ \phi\in
L^p(\O):\dint_{\O}\dint_{\O}|\phi(x)-\phi(y)|^pd\nu<+\infty\Big\}
$$
where $d\nu=\dyle\frac{dxdy}{|x-y|^{N+ps}}$. It is clear that $W^{s,p}(\O)$ is a Banach space endowed with the norm
$$
\|\phi\|_{W^{s,p}(\O)}=
\Big(\dint_{\O}|\phi(x)|^pdx\Big)^{\frac 1p}
+\Big(\dint_{\O}\dint_{\O}|\phi(x)-\phi(y)|^pd\nu\Big)^{\frac
1p}.
$$
In the same way we define the space $W^{s,p}_{0} (\O)$ as
the completion of $\mathcal{C}^\infty_0(\O)$ with respect to the
previous norm.
In the case where $\O=\ren$, we have the next Sobolev inequality
\begin{Theorem} \label{Sobolev}(Fractional Sobolev inequality)
Assume that $0<s<1$ and $p>1$ are such that $ps<N$. Then there exists a positive constant $S\equiv S(N,s,p)$ such that for all
$v\in C_{0}^{\infty}(\ren)$,
$$
\dint_{\mathbb{R}^{N}}\dint_{\mathbb{R}^{N}}
\dfrac{|v(x)-v(y)|^{p}}{|x-y|^{N+ps}}\,dxdy\geq S
\Big(\dint_{\mathbb{R}^{N}}|v(x)|^{p_{s}^{*}}\,dx\Big)^{\frac{p}{p^{*}_{s}}},
$$
where $p^{*}_{s}= \dfrac{pN}{N-ps}$.
\end{Theorem}
See \cite{Ponce} for a elementary proof.

We also will use the following extension result.
\begin{Lemma}\label{ext}
Assume that $\Omega\subset \ren$ is a regular domain, then for all
$w\in W^{s,p} (\O)$, there exists $\tilde{w}\in
W^{s,p}(\ren)$ such that $\tilde{w}_{|\Omega}=w$ and
$$
||\tilde{w}||_{W^{s,p} (\ren)}\le C
||w||_{W^{s,p}(\O)},
$$
where $C\equiv C(N,s,p,\O)>0$.
\end{Lemma}
See \cite{DPV} for the proof.

\begin{remark}\label{equiv}
If $\O$ is bounded regular domain, by the Poincaré inequality we
can endow $W^{s,p}_{0}(\O)$ with the equivalent norm
$$
|||\phi|||_{W^{s,p}_{0}(\O)}=
\Big(\dint_{\O}\dint_{\O}\dfrac{|\phi(x)-\phi(y)|^p}{|x-y|^{N+ps}}{dxdy}\Big)^{\frac
1p}.
$$
\end{remark}
For $w\in W^{s,p}(\ren)$, we  define the \textit{fractional p-Laplacian} as
$$
(-\Delta)^s_{p} w(x)=\mbox{ P.V. }
\dint_{\ren}\dfrac{|w(x)-w(y)|^{p-2}(w(x)-w(y))}{|x-y|^{N+ps}}{dy}.
$$

It is clear that for all $w, v\in W^{s,p}(\ren)$, we have
$$
\langle (-\Delta)^s_{p}w,v\rangle
=\dfrac 12\dint_{\ren}\dint_{\ren}\dfrac{|w(x)-w(y)|^{p-2}(w(x)-w(y))(v(x)-v(y))}{|x-y|^{N+ps}}{dxdy}.
$$
Now, if $w, v\in W^{s,p}_0(\O)$, we get
$$
\langle (-\Delta)^s_{p}w,v\rangle
=\dfrac 12\iint_{D_\O}\dfrac{|w(x)-w(y)|^{p-2}(w(x)-w(y))(v(x)-v(y))}{|x-y|^{N+ps}}{dxdy}.
$$
where $D_{\O}=\ren\times \ren\setminus \mathcal{C}\O\times
\mathcal{C}\O$.

It is easy to check that $(-\Delta)^{s}_{p}:W^{s,p}_0(\Omega)\longrightarrow W^{-s,p'}(\Omega)$. Notice that  $W^{-s,p'}(\Omega)$ is  the dual space of  $W^{s,p}_0(\Omega)$.

Let define now the corresponding parabolic spaces.

As in the local case, the space $L^{p}(0,T; W^{s,p}_0(\O))$ is defined as the set of function $\phi$ such that
$\phi\in L^p(\O_T)$ with $||\phi||_{L^{p}(0,T; W^{s,p}_0(\O))}<\infty$ where
$$
||\phi||_{L^{p}(0,T; W^{s,p}_0(\O))}=\Big(\int_0^T\iint_{D_\O}|\phi(x,t)-\phi(y,t)|^pd\nu\,dt\Big)^{\frac
1p}.
$$
It is clear that $L^{p}(0,T; W^{s,p}_0(\O))$ is a Banach spaces whose dual space is  $L^{p'}(0,T; W^{-s,p'}_0(\O))$.

For simplicity of typing and for any measurable function $u$, we set
$$
U(x,y,t)=|u(x,t)-u(y,t)|^{p-2}(u(x,t)-u(y,t)).
$$
We introduce the notions of solution to be use later.

\begin{Definition}\label{energy}
Assume  $(f,u_0)\in L^{p'}(0,T; W^{-s, p'}(\Omega))\times L^2(\O)$. We say that $u$ is an energy solution to problem \eqref{eq:def} if $u\in L^{p}(0,T; W^{s,p}_0(\O))\cap \mathcal{C}([0,T], L^2(\O))$, $u_t\in L^{p'}(0,T; W^{-s,p'}_0(\O))$, $u(x,.)\to u_0$ strongly in $L^2(\O)$ as $t\to 0$ and for all $v\in  L^{p}(0,T; W^{s,p}_0(\O))$ we have
\begin{equation*}
\int_0^T\langle u_t, v\rangle dt+ \dyle \dfrac 12\int_0^T\iint_{D_\O} U(x,y,t)(v(x,t)-v(y,t))d\nu\,dt=\iint_{\O_T}f(x,t)v dxdt
\end{equation*}
\end{Definition}
Notice that the existence of energy solution follows using classical argument for monotone operator. See \cite{Lio}.

For data $(f,u_0)\in L^1(\O_T)\times L^1(\O)$, we need to precise the sense in which the solution is defined.
\begin{Definition}\label{weak}
Assume  $(f,u_0)\in L^1(\O_T)\times L^1(\O)$, we say that $u$ is a weak solution (or distributional solution) to problem \eqref{eq:def} if for all $v\in  \mathcal{C}^\infty_0(\O_T)$ we have
\begin{equation*}
\dyle-\iint_{\O_T} u\,v_t dxdt +\dfrac 12\int_0^T\iint_{D_\O}U(x,y,t)(v(x,t)-v(y,t))d\nu\,dt=\iint_{\O_T} f(x,t)v dxdt.
\end{equation*}
\end{Definition}

In the local case a stronger notion of solution, entropy solution, is introduced in order to get uniqueness, see \cite{Pri}.
We  will extend this notion to the fractional framework in Section 4.
\begin{Definition}
We say that $u\in {\mathcal{T} }^{s,p}_0(\O_T)$ if $T_k(u)\in L^p(0,T; W^{s,p}_0(\O))$ for all $k>0$ where

\begin{equation}\label{f-trun}
T_k(s)=\left\{\begin{array}{cl}
s\,,&\hbox{ if }|s|\le k\,;\\[2mm]
k\dfrac{s}{|s|}\,,&\hbox{ if }|s|> k.
\end{array}\right.
\end{equation}
\end{Definition}
Some apriori estimates will be proved in the  classical  Marcinkiewicz space ${\mathcal{M}}^{q}(\O_T)$, that for the reader convenience, we define below.
\begin{Definition}
Let $u$ be a measurable function, define $$ \Phi_{u}(k)=\mu
\{(x,t)\in \O_T:|u(x,t)|>k\}.$$ We say
that $u$ is in the Marcinkiewicz space ${\mathcal{M}
}^{q}(\O_T, d\mu)$ if $ \Phi_{u}(k)\le Ck^{-q}$.

Notice that $L^{q}(\O_T)\subset {\mathcal{M}
}^{q}(\O_T)$ for all $q>1$.
\end{Definition}

The following elementary algebraic inequalities can be proved using suitable rescaling argument.
\begin{Lemma}\label{algg}
 Assume that $p\ge 1$, $(a, b) \in (\re^+)^2$ and $\a>0$, then there exist $c_1, c_2,c_3, c_4>0$, such that
\begin{equation}\label{alge1}
(a+b)^\a\le c_1a^\a+c_2b^\a
\end{equation}
and
\begin{equation}\label{alge3}
|a-b|^{p-2}(a-b)(a^{\a}-b^{\a})\ge c_3|a^{\frac{p+\a-1}{p}}-b^{\frac{p+\a-1}{p}}|^p.
\end{equation}
If moreover $\a\ge 1$, then under the same conditions on $a,b,p$ as above, we have
\begin{equation}\label{alge2}
|a+b|^{\a-1}|a-b|^{p}\le c_4 |a^{\frac{p+\a-1}{p}}-b^{\frac{p+\a-1}{p}}|^p,
\end{equation}
where $c_4>0$ is independent of $a$ and $b$.

Assume now that $a,b\in \re$ and $p\ge 1$, then
\begin{equation}\label{general}
|a-b|^{p-2}(a-b)(T_k(a)-T_k(b))\ge |T_k(a)-T_k(b)|^p
\end{equation}
and
\begin{equation}\label{general00}
|a-b|^{p-2}(a-b)(G_k(a)-G_k(b))\ge |G_k(a)-G_k(b)|^p
\end{equation}
where $G_k(s)=s-T_k(s)$.
\end{Lemma}
\section{Existence  of a weak solution}\label{sec1}

The main result of this section is the following.
\begin{Theorem}\label{th1}
Assume that $(f,u_0)\in L^1(\O_T)\times L^1(\O)$, then problem \eqref{eq:def} has a weak solution $u$ such that
$T_k(u)\in L^p(0,T; W^{s,p}_{ 0}(\Omega))$ for all $k>0$. Moreover, for all $q<\frac{N(p-1)+ps}{N+s}$ and for all
$s_1<s$, we have
\begin{equation}\label{estimm}
\int_0^T\iint_{\O\times \O}\dfrac{|u(x,t)-u(y,t)|^{q}}{|x-y|^{N+qs_1}}\,dy
\ dx\ dt\le M.
\end{equation}
If $p>2-\frac{s}{N}$, then $u\in L^p(0,T; W^{s_1,q}_{ 0}(\O))$ for all $1\le q<\frac{N(p-1)+ps}{N+s}$ and for all
$s_1<s$.
\end{Theorem}
To prove Theorem \ref{th1} we proceed by approximation.
Define $f_n=T_n(f)$ and $u_{0n}=T_n(u_0)$, then $(f_n,u_{0n})\in L^\infty(\O_T)\times L^\infty(\O)$ and $(f_n,u_{0n})\nearrow (f,u_0)$ strongly in $ L^1(\O_T)\times L^1(\O)$. Let $u_n$ be the unique solution to following approximated problem
\begin{equation}\label{pro:lineal1}
\left\{\begin{array}{rcll}
u_{nt}+(-\D^s_{p})u_n & = & f_n(x,t) &\mbox{ in }\O_T,\\
u_n & = & 0 & \hbox{  in } \ren\backslash\O\times (0,T),\\
u_n(x,0) & = & u_{0n}(x) & \mbox{ in }\O.
\end{array}
\right.
\end{equation}
Notice that the existence of $u_n$ follows using a direct modification of the classical result of \cite{Lio}.
Let us begin by proving the next a priori estimate.
\begin{Lemma}\label{lm:l1}
Consider the sequence $\{u_n\}_n$ defined as above, then $||u_n||_{\mathcal{M}^{p_1}(\O_T)}\le C$ for all $n$, where $p_1=p-1+\frac{ps}{N}$. In particular, for all $q<1+\frac{ps}{(p-1)N}$, we have
$$
||u^{p-1}_n||_{L^{q}(\O_T)}\le C\mbox{  for all }n.
$$
\end{Lemma}
\begin{proof}
Taking $T_k(u_n)$ as a test function in the problem \eqref{pro:lineal1}, it follows that
\begin{eqnarray*}
&\dyle \iint_{\O_T}u_{nt}T_k(u_n(x,t))\ dx\ dt+\iint_{\O_T}(-\D^s_{p})\, u_n(x,t)[T_k(u_n(x,t))]\ dx\ dt\\
&=\dyle \iint_{\O_T} f_n(x,t)[T_k(u_n(x,t))] \, dx\ dt\le Ck.
\end{eqnarray*}
Integrating by part, we reach that
\begin{eqnarray*}
&\dyle \int_{\O}\Theta _k(u_n(x,T))\ dx+\dyle \frac{1}{2}\int_0^T\iint_{D_\O}\,U_n(x,y,t)[T_k(u_n(x,t))-T_k(u_n(y,t))]d\nu\, dt\\
& \le \dyle ck+\int_{\O}\Theta _k(u_{0n}(x))\ dx\le c k+k \int_{\O}|u_{0}(x)|\ dx\\
&\le C_1k,
\end{eqnarray*}
where $\theta(\s)=\dyle\int_0^\sigma T_k(\tau)d\tau$.

Thus, using inequality \eqref{general} and the above estimate, it follows that
$$
\sup_{t\in [0,T]}\dyle \int_{\O}\Theta _k(u_n(x,t))\ dx\\
+\dyle \frac{1}{2}\int_0^T\iint_{D_\O}
|T_k(u_n(x,t))-T_k(u_n(y,t))|^{p}d\nu\,dt\le M k.$$
Then, up to a subsequence, there exists a measurable function $u$ such that $T_k(u_n)\rightharpoonup T_k(u)$ weakly in $L^p((0,T); W^{s,p}_0(\Omega))$ and $u_n\to u$ a.e in $\O_T$.

By the Sobolev inequality, we get
\begin{eqnarray*}
& \dyle \int_{0}^T\Big(\int_{\O}|T_k(u_n(x,t))|^{p^*_{s}}\ dx\Big)^{\frac{p}{p^*_s}}\ dt \\
&\le \dyle \int_0^T\iint_{D_\O}|T_k(u_n(x,t))-T_k(u_n(y,t))|^{p}d\nu\ dt \le M k.
\end{eqnarray*}
Hence
  $$\int_{0}^T \left(\dyle \int_{\O}
|T_k(u_n(x,t))|^{p^*_{s}} dx\right)^{\frac{1}{p^*_{s}}}\ dt \le
C(Mk)^{1/p}. $$
Let $1<r<p^*_s$ and define $r_1=\big(\frac{p^*_{s}}{p^*_{s}-1}\big)(r-1) $, $r_2=1-\frac{r_1}{p^*_{s}}$, where  $r=r_1+r_2$.

Fix $t_1<T$, then
\begin{eqnarray*}
&\dyle \int_{0}^{t_1}\int_{\O}|T_k(u_n(x,t))|^{r}\ dx \ dt \le  \int_{0}^{t_1}\int_{\O}|T_k(u_n(x,t))|^{r_1}|u_n(x,t)|^{r_2}\ dx\ dt\\
&\le \dyle \int_{0}^{t_1}\Big(\int_{\O}|T_k(u_n(x,t))|^{p^*_{s}}\ dx\Big)^{\frac{r_1}{p^*_s}}\Big(\int_{\O}|u_n(x,t)|^{r_2(\frac{p^*_{s}}{r_1})'}\ dx\Big)^{1-\frac{r_1}{p^*_s}}\ dt\\
&\dyle \le \int_{0}^{t_1}\Big(\int_{\O}|T_k(u_n(x,t))|^{p^*_{s}}\ dx\Big)^{\frac{r_1}{p^*_s}}\Big(\int_{\O}|u_n(x,t)|^{r_2\frac{p^*_{s}}{p^*_{s}-r_1}}\ dx\Big)^{1-\frac{r_1}{p^*_s}}\ dt\\
&\dyle \le\Big(\underset{t\in[0,T]}{\sup }\int_{\O}|u_n(x,t)| dx \Big)\int_{0}^T\Big(\int_{\ren}|T_k(u_n(x,t))|^{p^*_{s}}\ dx\Big)^{\frac{r_1}{p^*_s}} dt\\
&\dyle \le c \int_{0}^T\Big(\int_{\ren}|T_k(u_n(x,t))|^{p^*_{s}}\ dx\Big)^{\frac{r_1}{p^*_s}}\ dt\le  c Mk^{\frac{r_1}{p}}\le C k^{\frac{r_1}{p}}.
\end{eqnarray*}
Thus
$$
\dyle \iint_{\O_T}|T_k(u_n(x,t))|^{r}\ dx \ dt \le C k^{\frac{r_1}{p}}.
$$
Now, using the fact that $|\{|u_n|>k\}|=|\{|T_k(u_n)|=k\}|$, we obtain that
$$ k^{r}T\Phi_u\{(x,t)\in \O_T
:|u|>k\}\le\iint_{\O_T}|T_k(u_n(x,t))|^{r}\ dx \ dt \le TC k^{\frac{r_1}{p}}.$$
Hence
 $$\Phi_u\{(x,t)\in \O_T
:|u_n|>k\}\le C k^{-(r-\frac{r_1}{p})}\le C k^{-\a},$$
where $\a=1+r_1[\frac{p^*_{s}(p-1)-p}{pp^*_{s}}]$. Letting $r_1\rightarrow p$, it follows that $\a\to 1+[\frac{p^*_{s}(p-1)-p}{p^*_{s}}]= [\frac{p(p^*_{s}-1)}{p^*_{s}}]$.

Thus
$ \Phi_u\{(x,t)\in \O_T :|u_n|>k\}\le C M^{\frac{p^*_{s}}{p}}k^{-p_1}$ where $p_1=p-1+\frac{ps}{N}$. Hence $||u_n||_{\mathcal{M}^{p_1}(\O_T)}\le C$ for all $n$, and the result follows.

 By the previous estimates and using the Vitali lemma it holds that $u^{p-1}_n\to u^{p-1}$ strongly in $L^q(\O_T)$ for all $q<1+\frac{ps}{(p-1)N}$.
\end{proof}

\

We prove now that the sequence $\{u_n\}_n$ is bounded in a suitable fractional Sobolev spaces, more precisely we have the following result.
\begin{Lemma}\label{lm:l11}
Let $\{u_n\}_n$ defined as above, then for all $q<p_2=\frac{N(p-1)+ps}{N+s}$ and for all $s_1<s$, we have
\begin{equation}\label{eq:eq1}
\int_0^T\iint_{\O\times \O}\dfrac{|u_n(x,t)-u_n(y,t)|^{q}}{|x-y|^{N+qs_1}}d\nu\ dt\le M.
\end{equation}
In particular, if $p>\frac{2N+s}{N+s}$, then $\{u_n\}_n $ is bounded in $ L^q(0,T; W^{s_1,q}_0(\O))$ for all
$1<q<p_2=\frac{N(p-1)+ps}{N+s}$.
\end{Lemma}
\begin{proof} In what follows, we denote by $C,C_1, C_2,...,$ any positive constants that are independent of $\{u_n\}_n$ and can
change from one line to another.

We follow closely the argument used in \cite{AAB}. Define $$w_n(x,t)=1-\dfrac{1}{(u^+_n(x,t)+1)^{\a}},$$ where $\a>0$ to be chosen later, then using  $w_n$ as a test function in \eqref{pro:lineal1}, we get

\begin{eqnarray*}
&\dyle \int_0^T
\int_{\O} u_{nt}w_n(x,t)\ dx\ dt+\int_0^T\iint_{D_{\O}}\,U_n(x,y,t)w_n(x,t)d\nu\, dt\\
&\le\dyle\iint_{\O_T}|f_n(x,t)|w_n\, dx\,dt\le C.
\end{eqnarray*}
Integrating by part we  find that,
\begin{eqnarray*}
&\dyle \dyle\iint_{\O_T} u_{nt}(x,t)w_n(x,t)\ dx\ dt=\int_{\O} u^+_{n}(x,T)\ dx-\int_{\O} u^+_{0n}(x)dx\\
&+\dyle \frac{1}{1+\a}\int_{\O}\Big[\frac{1}{(u^+_n(x,T)+1)^{\a+1}}-\frac{1}{(u^+_{n0}(x)+1)^{\a+1}}\Big]\ dx\mbox{  if }\a\neq 1,
\end{eqnarray*}
and \begin{eqnarray*}
&\dyle \dyle\iint_{\O_T} u_{nt}(x,t)w_n(x,t)\ dx\ dt=\int_{\O} u^+_{n}(x,T)\ dx-\int_{\O} u^+_{0n}(x)\ dx \\
&+\dyle \int_{\O}\Big[\log(u^+_n(x,T)+1)-\log(u^+_{n0}(x)+1)\Big]\ dx\mbox{  if  }\a=1.
\end{eqnarray*}
Hence, in any case, since $\dyle\underset{t\in[0,T]}{\sup }\int_{\O}|u_n(x,t)| dx\le C$ for all $n$, it follows that
\begin{eqnarray*}
&\dyle \dyle\iint_{\O_T} u_{nt}(x,t)w_n(x,t)\ dx\ dt\ge \int_{\O} u^+_{n}(x,T)\ dx-C.
\end{eqnarray*}
We deal now with the term
$$
\int_0^T\iint_{D_{\O}}\,|u_n(x,t)-u_n(y,t)|^{p-2}(u_n(x,t)-u_n(y,t))w_n(x,t)d\nu\,dt.
$$
Let $v_n=u^+_n+1 $ and define $\tilde{V}_n(x,y,t)=(v_n(x,t)-v_n(y,t))$. Taking into consideration that
\begin{eqnarray*}
&|u_n(x,t)-u_n(y,t)|^{p-2}(u_n(x,t)-u_n(y,t))\Big((u^+_n(x,t)+1)^{\a}-(u^+_n(y,t)+1)^{\a}\Big)\ge
\\
&
|u^+_n(x,t)-u^+_n(y,t)|^{p-2}(u^+_n(x,t)-u^+_n(y,t))\Big((u^+_n(x,t)+1)^{\a}-(u^+_n(y,t)+1)^{\a}\Big)=\\
&|\tilde{V}_n(x,y,t)|^{p-2}\tilde{V}_n(x,y,t)\Big(v^\a_n(x,t)-v^\a_n(y,t)\Big),
\end{eqnarray*}
it follows that
\begin{equation}\label{R3}
\begin{array}{lll}
& \dyle \int_0^T\iint_{D_{\O}}\,|u_n(x,t)-u_n(y,t)|^{p-2}(u_n(x,t)-u_n(y,t))w_n(x,t)d\nu\,dt\ge \\
&\dyle \int_0^T\iint_{D_{\O}} \,|\tilde{V}_n(x,y,t)|^{p-2}\tilde{V}_n(x,y,t)\bigg(\frac{v_n^{\a}(x,t)-v_n^{\a}(y,t)}{v_n^{\a}(y,t)v_n^{\a}(x,t)}\bigg)d\nu\ dt.
\end{array}
\end{equation}
Using inequality \eqref{alge3} and by \eqref{R3}, it holds,
\begin{equation}\label{R5}
\int_0^T\iint_{D_\O} \,\frac{|v_n^{\frac{p+\a-1}{p}}(x,t)-v_n^{\frac{p+\a-1}{p}}(y,t)|^{p}}{v_n^{\a}(y,t)v_n^{\a}(x,t)}d\nu\ dt\le C.
\end{equation}
Fix $q<p_2$ and $s<s_1$, then there exists $q_1<q$ such that $s_1=\frac{q_1}{q}s$. Therefore we get
\begin{eqnarray*}
&\dyle \int_0^T\iint_{\O\times\O}\dfrac{|v_n(x,t)-v_n(y,t)|^{q}}{|x-y|^{N+qs_1}}dxdydt\\ \\
&\dyle =\dyle \int_0^T\iint_{\O\times\O}\dfrac{|\tilde{V}_n(x,y,t)|^{q}}{|x-y|^{qs}}
\frac{(v_n(x,t)+v_n(y,t))^{\a-1}}{(v_n(x,t)v_n(y,t))^{\a}}\frac{(v_n(x,t)v_n(y,t))^{\a}}{(v_n(x,t)+v_n(y,t))^{\a-1}}\dfrac{|x-y|^{(q-q_1)s}}{|x-y|^{N}}dxdydt\\ \\
&\dyle \le\Big(\int_0^T\iint_{\O\times\O}\dfrac{|\tilde{V}_n(x,y,t)|^{p}(v_n(x,t)+v_n(y,t))^{\a-1}}{|x-y|^{N+ps}(v_n(x,t)v_n(y,t))^{\a}}dxdydt\Big)^{\frac{q}{p}}\times \\ \\
&\dyle \Big(\int_0^T\iint_{\O\times\O}\frac{(v_n(x,t)+v_n(y,t))^{\a-1}}{(v_n(x,t)v_n(y,t))^{\a}}\frac{(v_n(x,t)+v_n(y,t))^{\frac{p\a}{p-q}}}{(v_n(x,t)+v_n(y,t))^{\frac{p(\a-1)}{p-q}}}\dfrac{|x-y|^{\frac{p(q-q_1)s}{p-q}}}{|x-y|^{N}}dxdydt\Big)^{\frac{p-q}{p}}\\ \\
&\dyle \le\Big(\int_0^T\iint_{\O\times\O}\dfrac{|\tilde{V}_n(x,y,t)|^{p}\Big(v_n(x,t)+v_n(y,t)\Big)^{\a-1}}{|x-y|^{N+ps}(v_n(x,t)v_n(y,t))^{\a}}dxdydt \Big)^{\frac{q}{p}}\times \\ \\
&\dyle \Big(\int_0^T\iint_{\O\times\O}\Big(\frac{(v_n(x,t)v_n(y,t))^{\a}}{(v_n(x,t)+v_n(y,t))^{\a-1}}\Big)^\frac{q}{p-q}\dfrac{|x-y|^{\frac{p(q-q_1)s}{p-q}}dx\ dy\ dt}{|x-y|^{N}}\Big)^{\frac{p-q}{p}}.
\end{eqnarray*}
Using inequality \eqref{alge2} and by \eqref{R5}, it follows that
\begin{eqnarray*}
& \dyle \Big(\int_0^T\iint_{\O\times\O}\dfrac{|\tilde{V}_n(x,y,t)|^{p}(v_n(x,t)+v_n(y,t))^{\a-1}}{|x-y|^{N+ps}(v_n(x,t)v_n(y,t))^{\a}} dxdydt\Big)^{\frac{q}{p}}\\
&\dyle \le\Big(\int_0^T\iint_{\O\times\O}\frac{|v_n^{\frac{p+\a-1}{p}}(x,t)-v_n^{\frac{p+\a-1}{p}}(y,t)|^{p}}{|x-y|^{N+ps}(v_n(x,t)v_n(y,t))^{\a}} dx\ dy\ dt\Big)^{\frac{q}{p}}\le C.
\end{eqnarray*}
So we obtain
\begin{eqnarray*}
& \dyle \int_0^T\iint_{\O\times\O}\dfrac{|\tilde{V}_n(x,y,t)|^{q}}{|x-y|^{N+qs_1}}dx\ dy\ dt\\
&\dyle \le
c\Big(\int_0^T\iint_{\O\times\O}\Big(\frac{v_n(x,t)v_n(y,t)}{v_n(x,t)+v_n(y,t)}\Big)^{\a\frac{q}{p-q}}\dfrac{(v_n(x,t)+v_n(y,t))^\frac{q}{p-q}}{|x-y|^{N-\frac{ps(q-q_1)}{p-q}}} dx \ dy\ dt\Big)^{\frac{p-q}{p}}.
\end{eqnarray*}
Using inequality \eqref{alge1}, we reach that
\begin{eqnarray*}
(v_n(x,t)+v_n(y,t))\bigg( \frac{v_n(x,t)v_n(y,t)}{v_n(x,t)+v_n(y,t)}\bigg)^{\a} &\le &  C_1 (v_n(x,t)+v_n(y,t))^{\a+1}\\
&\le & C_1v_n^{\a+1}(x,t)+C_2v_n^{\a+1}(y,t).
\end{eqnarray*}
Then,
\begin{equation}\label{main00}
\begin{array}{lll}
&\dyle \int_0^T\iint_{\O\times\O}\dfrac{|\tilde{V}_n(x,y,t)|^{q}}{|x-y|^{N+qs_1}}dx \ dy\ dt\\
&\dyle \le  c_1\Big(\int_0^T\iint_{\O\times\O}\dfrac{v_n^{\frac{(\a+1)q}{p-q}}(x,t) dx \ dy\ dt}{|x-y|^{N-\frac{ps(q-q_1)}{p-q}}}\Big)^{\frac{p-q}{p}} +  c_2\Big(\int_0^T\iint_{\O\times\O}\dfrac{v_n^{\frac{(\a+1)q}{p-q}}(y,t) dx \ dy\ dt}{|x-y|^{N-\frac{ps(q-q_1)}{p-q}}}\Big)^{\frac{p-q}{p}}.
\end{array}
\end{equation}
Since $\Omega$ is a bounded domain, we get the existence of $R>0$ such that $\Omega\subset\subset B_R(0)$. Hence
$$\int_0^T\iint_{\O\times\O}\dfrac{v^{\frac{(\a+1)q}{p-q}}_n(x,t)dx \ dy\, dt}{|x-y|^{N-\frac{ps(q-q_1)}{p-q}}}\le\int_0^T \int_{B_R(0)} v^{\frac{(\a+1)q}{p-q}}_n(x,t)dx\,dt\int_{B_R(0)}\dfrac{dy}{|x-y|^{N-\frac{ps(q-q_1)}{p-q}}}$$

To compute the last integral, we follow closely the radial computations in \cite{FV} and \cite{G}.
We set $r=|x|$ and $\rho=|y|$, then $x=rx', y=\rho y'$,
where $|x'|=|y'|=1$. Define $\kappa=\frac{(\a+1)q}{p-q}$ and $\theta=\frac{ps(q-q_1)}{p-q}$, it follows that
$$\int_0^T\iint_{\O\times\O}\dfrac{v_n^{\kappa}(x,t)dxdydt}{|x-y|^{N-\theta}}\le
\int_0^T\int_{B_R(0)}
v_n^{\kappa} (x,t) \ dx\ dt
\dint\limits_0^{R}\dfrac{\rho^{N-1}}{
r^{N-\theta}}\left(
\dint\limits_{|y'|=1}\dfrac{dH^{n-1}(y')}{|x'-\frac{\rho}{r}
y'|^{N-\theta}} \right) \,d\rho.$$
Setting $\sigma=\dfrac{\rho}{r}$, then

$$\int_0^T\iint_{\O\times\O}\dfrac{v_n^{\kappa}(x,t)dx \ dy\ dt}{|x-y|^{N-\theta}}\le
\int_0^T\int_{B_R(0)}
 v_n^{\kappa}(x,t)|x|^{\theta}\ dx\ dt \dint\limits_0^{\frac{R}{r}}\sigma^{N-1}
K_\theta(\sigma)\,d\sigma.
$$
where
\begin{equation}\label{kkk}
K_\theta(\sigma)=\dint\limits_{|y'|=1}\dfrac{dH^{n-1}(y')}{|x'-\s
y'|^{N-\theta}}=2\frac{\pi^{\frac{N-1}{2}}}{\beta(\frac{N-1}{2})}\int_0^\pi
\frac{\sin^{N-2}(\xi)}{(1-2\sigma \cos
(\xi)+\sigma^2)^{\frac{N-\theta}{2}}}d\xi.
\end{equation}
Therefore we conclude that
\begin{equation}\label{RTV}
\begin{array}{lll}
&\dyle \int_0^T\iint_{\O\times\O}\dfrac{v_n^{\kappa}(x,t)dx \ dy\ dt}{|x-y|^{N-\theta}}\le\int_0^T\int_{B_\frac{R}{3}(0)}
 v_n^{\kappa}(x,t)|x|^{\theta}\ dx\ dt \dint\limits_0^{\frac{R}{r}}\sigma^{N-1}
K_\theta(\sigma)\,d\sigma\\
&\dyle +\int_0^T\int_{B_R(0)\backslash B_\frac{R}{3}(0)}
 v_n^{\kappa}(x,t)|x|^{\theta}\ dx\ dt \dint\limits_0^{\frac{R}{r}}\sigma^{N-1}
K_\theta(\sigma)\,d\sigma.
\end{array}
\end{equation}
Recall that $r=|x|$, then if $x\in B_R(0)\backslash B_\frac{R}{3}(0)$, it holds $\frac{R}{r}<3$. Hence taking into consideration that $\theta>0$ and the behavior of $K_\theta$ near $1$, we reach that
$$
\dint\limits_0^{\frac{R}{r}}\sigma^{N-1}
K_\theta(\sigma)\,d\sigma\le \dint\limits_0^{3}\sigma^{N-1}
K_\theta(\sigma)\,d\sigma=C_1.
$$
Now, for $x\in  B_\frac{R}{3}(0)$,
\begin{eqnarray*}
\dint\limits_0^{\frac{R}{r}}\sigma^{N-1}
K_\theta(\sigma)\,d\sigma & = & \dint\limits_0^{3}\sigma^{N-1}
K_\theta(\sigma)\,d\sigma + \dint\limits_3^{\frac{R}{r}}\sigma^{N-1}
K_\theta(\sigma)\,d\sigma= C_1 + \dint\limits_3^{\frac{R}{r}}\sigma^{N-1}
K_\theta(\sigma)\,d\sigma\\
&\le & C_1+(\frac{R}{r})^a \dint\limits_3^{\frac{R}{r}}\sigma^{N-1-a}
K_\theta(\sigma)\,d\sigma\
\end{eqnarray*}
where $a>0$ to be chosen later.  It is clear that
$$
\dint\limits_3^{\frac{R}{r}}\sigma^{N-1-a}
K_\theta(\sigma)\,d\sigma\le \dint\limits_3^{\infty}\sigma^{N-1-a}
K_\theta(\sigma)\,d\sigma.
$$
Choosing $a>\theta$, it follows that $\dint\limits_3^{\infty}\sigma^{N-1-a}
K_\theta(\sigma)\,d\sigma=C_2<\infty$. Now, going back to \eqref{RTV}, it holds
\begin{equation}\label{RTV0}
\begin{array}{lll}
&  \dyle \int_0^T\iint_{\O\times\O}\dfrac{v_n^{\kappa}(x,t)dx \ dy\ dt}{|x-y|^{N-\theta}}\\ & \le
\dyle C_1 \int_0^T\int_{B_R(0)}
 v_n^{\kappa}(x,t)|x|^{\theta}\ dx\ dt
\dyle +C_2R^a\int_0^T\int_{B_\frac{R}{3}(0)}
 v_n^{\kappa}(x,t)|x|^{\theta-a}\ dx\ dt
\end{array}
\end{equation}
Recall that $\kappa=\frac{(\a+1)q}{p-q}$, since $q<\frac{(p-1)N+ps}{N+s}$, we can choose $\a>0$ such that $\kappa<p-1+\frac{ps}{N}$.
Hence, taking into consideration the result of Lemma \ref{lm:l1}, choosing $a$ very close to $\theta$ and using H\"older inequality, we deduce that
$$
\int_0^T\iint_{\O\times\O}\dfrac{v_n^{\frac{(\a+1)q}{p-q}}(x,t) dx \ dy\ dt}{|x-y|^{N-\frac{ps(q-q_1)}{p-q}}}\le C.
$$
In a symmetric way, we can show that
$$
\int_0^T\iint_{\O\times\O}\dfrac{v_n^{\frac{(\a+1)q}{p-q}}(y,t) dy \ dx\ dt}{|x-y|^{N-\frac{ps(q-q_1)}{p-q}}}\le C\mbox{  for all  }n.$$
Going back to \eqref{main00} and taking into consideration the previous estimates, we conclude that
\begin{equation}\label{last001}
\begin{array}{lll}
& \dyle\int_0^T\iint_{{\O\times\O}}\dfrac{|u^+_n(x,t)-u^+_n(y,t)|^{q}}{|x-y|^{N+qs_1}} dy \ dx\ dt\\
&\dyle=
\int_0^T\iint_{{\O\times\O}}\dfrac{|v_n(x,t)-v_n(y,t)|^{q}}{|x-y|^{N+qs_1}} dy \ dx\ dt\le C.
\end{array}
\end{equation}
In the same way and using $\Big(1-\dfrac{1}{(u^-_n(x,t)+1)^{\a}}\Big)$ as a
test function in \eqref{pro:lineal1}, it follows that
\begin{equation}\label{last002}
\dyle\int_0^T\iint_{{\O\times\O}}\dfrac{|u^-_n(x,t)-u^-_n(y,t)|^{q}}{|x-y|^{N+qs_1}} dy \ dx\ dt\le C.
\end{equation}
Combining the estimates \eqref{last001} and \eqref{last002}, we reach that
$$\int_0^T\int_\O\int_{\O}\dfrac{|u_n(x,t)-u_n(y,t)|^{q}}{|x-y|^{N+qs_1}}dx\,dy\le C.$$
Hence we conclude.
\end{proof}

To prove that $u\in \mathcal{C}([0,T], L^1(\O))$, we need the next lemma.
\begin{Lemma}\label{tt}
Let $\{u_n\}_n$ be defined as above, then $\{u_n\}_n$ converge strongly to $u$ in $C([0,T], L^1(\O))$.
\end{Lemma}
\begin{proof}
Let $m, n\in \ene$, then for all $\phi \in L^p([0,T]; W^{s,p}_0(\O))$,
\begin{eqnarray*}
&\dyle \iint_{\O_T} (u_n-u_m)_t(x,t)\phi(x,t) dx\ dt\\
&+\dyle \iint_{\Omega_T}\langle (-\D^s_{p})\, u_n(x,t)-(-\D^s_{p})\, u_m(x,t),
\phi(x,t)\rangle dxdt\\
&\dyle= \iint_{\O_T} (f_n-f_m)\phi dxdt.
\end{eqnarray*}
Let $\phi(x,t)=T_1(u_n-u_m)_{[0,t]}(x,t)$, with $t\leq T$, setting $\O_t=\O\times (0,t)$, we get
\begin{eqnarray*}
&\dyle \iint_{\O_t}\langle (u_n-u_m)_\tau(x,\tau) ,T_1(u_n-u_m)(x,\tau)\rangle dx\ d\tau\\
&\dyle +\iint_{\O_t}\langle(-\D^s_{p})\, u_n(x,\tau)- (-\D^s_{p})\, u_m(x,\tau),
T_1(u_n-u_m)\rangle dxd\tau\\
&\dyle= \iint_{\O_t}(f_n-f_m)(x,\tau)T_1(u_n-u_m)(x,\tau)\,dx\,d\tau\leq\iint_{\O_T} |f_n-f_m|  dxd\tau.
\end{eqnarray*}
It is clear that
$$\iint_{\O_T}\langle (u_n-u_m)_\tau(x,\tau) ,T_1(u_n-u_m)(x,\tau)\rangle dx\ d\tau= \io [\Theta_1(u_n-u_m)]^t_0(x,\tau)\,dx.  $$
Now, by inequality \eqref{general} we obtain that
$$\iint_{\Omega_t}\langle (-\D^s_{p})\, u_n(x,\tau)-(-\D^s_{p})\, u_m(x,\tau),
T_1(u_n-u_m)\rangle dxd\tau\ge 0.$$
Thus
$$\io [\Theta_1(u_n-u_m)](x,t)dx\le \io [\Theta_1(u_n-u_m)](x,0)dx+ \iint_{\O_T}|f_n-f_m|dx\,d\tau.$$
Recall that $\Theta_1(\s)\le |\s|$, thus, for all $t\le T$,
$$\io [\Theta_1(u_n-u_m)](x,t) \leq \io |u_{0n}-u_{0m}|dx +  \iint_{\O_T}|f_n-f_m|dx\,d\tau.$$
Denote $b_{n,m}$ the right hand side, thus
$$\int_{|u_n-u_m|<1}|u_n-u_m|^2(x,t)dx+ \int_{|u_n-u_m|>1}|u_n-u_m|(x,t)dx\le 2b_{n,m}.$$
Since
\begin{eqnarray*}
\int_{\O_T}|u_n-u_m|(x,t)dx & = & \int_{|u_n-u_m|<1}|u_n-u_m|(x,t)dx+ \int_{|u_n-u_m|>1} |u_n-u_m|(x,t)dx\\
&\le & \Big( \int_{|u_n-u_m|<1}|u_n-u_m|^2(x,t)dx\Big)^{\frac 12}|\O_T|^{\frac 12}+ 2b_{n,m}\\
&\le & |2\O_T|^{\frac 12}b_{n,m}^{\frac 12}+2b_{n,m},
\end{eqnarray*}
taking into consideration that the sequences $\{f_n\}_n$ and $\{u_{0n}\}_n$ converge strongly in $L^1(\O_T)$ and $L^1(\O)$ respectively, we conclude that $b_{n,m}\to 0$ for $n,m \to \infty$.

Therefore we conclude that $\{u_n\}_n$ is a Cauchy sequence in $C([0,T], L^1(\O))$ and then $u_n\to u$ in $C([0,T], L^1(\O))$.
\end{proof}

We summarize the previous Lemmas as follows:
\begin{itemize}
\item   $u\in C([0,T], L^1(\O))$,
\item $T_k(u)\in L^p(0,T; W^{s,p}_{ 0}(\Omega))$, $u^{p-1}\in
L^\s(\O_T)$ for all $\s<\frac{N(p-1)+ps}{N(p-1)}$ and
\item $T_k(u_n)\rightharpoonup T_k(u)$ weakly in $L^p(0,T; W^{s,p}_{ 0}(\Omega))$.
\end{itemize}
It is clear that $u_n\to u$ a.e. in $\O_T$, then, since $u_n=0$ a.e. in $(\ren\setminus\O) \times (0,T)$, we get $u=0$ a.e. in $(\ren\setminus\O) \times (0,T)$.

Recall that
$$
U_n(x,y,t)=|u_n(x,t)-u_n(y,t)|^{p-2}(u_n(x,t)-u_n(y,t))\mbox{  and  }$$
$$U(x,y,t)=|u(x,t)-u(y,t)|^{p-2}(u(x,t)-u(y,t)).
$$
Since $\O$ is a bounded domain, then by the result of Lemma \ref{lm:l1} and using Vitali's Lemma,  we reach that
$$
U_n\to U \mbox{   strongly in   }L^1((\O\times \O)\times (0,T), d\nu\,dt).
$$

{\bf Proof of Theorem \ref{th1}.}

Let $\phi\in \mathcal{C}^\infty_0(\Omega_T)$ and define
$\Phi(x,y,t)=\phi(x,t)-\phi(y,t)$, taking $\phi$ as a test function in \eqref{pro:lineal1}, it follows that
\begin{equation}\label{TR}
\begin{array}{lll}
&\dyle \int\int_{\O_T}u_{nt} \phi(x,t)\,dx\,dt+
\frac 12\int_0^T\iint_{D_{\O}} \,U_n(x,y,t)\Phi(x,y,t)d\nu\,dt \\
&=\dyle \int\int_{\O_T} f_n(x,t)\phi(x,t) \, dx\, dt.
\end{array}
\end{equation}
It is clear that
$$\iint_{\O_T}u_{nt} \phi(x,t)\,dx\,dt=-\iint_{\O_T}u_{n} \phi_t(x,t)\,dx\,dt.$$
Hence
$$-\iint_{\O_T}u_{n} \phi_t(x,t)\,dx\,dt\to-\iint_{\O_T}u \phi_t(x,t)\,dx\,dt
 \text{ \ as n} \to \infty.$$
In the same way we have
$$
\iint_{\O_T} f_n(x,t)\phi(x,t) \, dx\,dt\to \iint_{\O_T} f(x,t)\phi(x,t)\, dx\,dt \mbox{   as   }n\to \infty.
$$
We claim that
\begin{equation}\label{esti01}
\int_0^T\iint_{D_{\O}}\Big(U_n(x,y,t)-U(x,y,t)\Big)\Phi(x,y,t)d\nu\,dt\to 0\mbox{   as    }n\to \infty.
\end{equation}
Since $u_n\to u$ a.e. in $\Omega_T$, then
$$
\dfrac{U_n(x,y,t)\Phi(x,y,t)}{|x-y|^{N+ps}}\to \dfrac{U(x,y,t)\Phi(x,y,t)}{|x-y|^{N+ps}}\:a.e. \mbox{  in  }D_{\Omega_T}\equiv D_\O\times (0,T).
$$
Using the fact that $u(x,t)=u_n(x,t)=\phi(x,t)=0$ for all $x\in (\ren\backslash\O)\times (0,T)$, we reach that
$$
\int_0^T\int_{\ren\backslash \O}\int_{\ren\backslash \O} (U_n(x,y,t)-U(x,y,t))\Phi(x,y,t)d\nu\, dt=0.
$$
Thus
\begin{eqnarray*}
& \dyle\int_0^T\iint_{D_{\O}}(U_n(x,y,t)-U(x,y,t))\Phi(x,y,t)d\nu\, dt\\
&=\dyle \int_0^T\iint_{\O\times \O}(U_n(x,y,t)-U(x,y,t))\Phi(x,y,t)d\nu\,dt\\
&+\dyle \int_0^T\int_{\ren\backslash \O}\int_{\O}(U_n(x,y,t)-U(x,y,t))\Phi(x,y,t)d\nu\,dt\\
&\dyle +\int_0^T\int_{\O}\int_{\ren\backslash \O} (U_n(x,y,t)-U(x,y,t))\Phi(x,y,t)d\nu\,dt\\
&=I_1+I_2+I_3.
\end{eqnarray*}
Since $U_n\to U \mbox{   strongly in   }L^1((\O\times \O)\times (0,T), d\nu dt)$, then $I_1\to 0\mbox{   as   }n\to \infty.$

We deal now with $I_2$. It is clear that in $(\O\times B_R\backslash \O)\times(0,T)$, we have
$$
|(U_n(x,y,t)-U(x,y,t))\Phi(x,y,t)|\le (|u_n(x,t)|^{p-1}+|u(x,t)|^{p-1})|\phi(x,t)|.
$$
Since
$$
\sup_{\{x\in \text{Supp}\phi, \:\:y\in B_R\backslash \O\}}\dfrac{1}{|x-y|^{N+ps}}\le C,
$$
then
\begin{eqnarray*}
\Big|\dfrac{(U_n(x,y,t)-U(x,y,t))\Phi(x,y,t)}{|x-y|^{N+ps}}\Big| &\le & C(|u_n(x,t)|^{p-1}+|u(x,t)|^{p-1})|\phi(x,t)|\\
&\equiv & Q_n(x,y,t).
\end{eqnarray*}
Notice that $Q_n\to Q$ strongly in $L^1((\O\times B_R\backslash \O)\times(0,T))$ with $$Q(x,y,t)=2|u(x,t)|^{p-1}|\phi(x,t)|.$$ Therefore, using the Dominated convergence Theorem we reach that $I_2\to 0\mbox{   as   }n\to \infty$.
In the same way we obtain that $I_3\to 0\mbox{   as   }n\to \infty$. Hence the claim follows.

As a conclusion and passing to the limit in \eqref{TR} there results that
\begin{eqnarray*}
&-\dyle \iint_{\O_T}u \phi_t(x,t)\,dx\,dt+\frac 12\int_0^T\iint_{{\O_T}} \,U(x,y,t)\Phi(x,y,t)d\nu\,dt\\
&=\dyle\iint_{\O_T} f(x,t)\phi(x,t) \, dx\, dt.
\end{eqnarray*}
Hence we conclude. \cqd

\begin{remark}
The same existence result holds also if $(f,u_0)\in  \mathfrak{M}^1(\O_T)\times \mathfrak{M}^1(\O)$, the set of Radon measures on $\O_T$ and $\O$ respectively.
\end{remark}
\section{Nonnegative solutions obtained as limit of approximation are entropy solutions}\label{s2}
We state now the definition of entropy solution inspired from \cite{Pri}.
\begin{Definition}\label{def1}
Let $(f,u_0)\in L^1(\O_T)\times L^1(\O)$ be nonnegative functions.
We say that $u\in C([0,T];L^1(\O))$ is an entropy solution to the problem
\eqref{eq:def} if $u\in {\mathcal{T} }^{s,p}_0(\O_T)$ and
\begin{enumerate}
\item Setting
\begin{equation}\label{regi}
\begin{array}{lll}
R_h & = & \bigg\{(x,y,t)\in \re^{2N}\times (0,T): h+1\le \max\{|u(x,t)|,|u(y,t)| \}\\
& & \mbox{  with  } \min\{|u(x,t)|,|u(y,t)| \}\le h\mbox{  or  }u(x,t)u(y,t)<0\bigg\}
\end{array}
\end{equation}
then
\begin{equation}\label{entro001}
\iiint_{R_h}|u(x,t)-u(y,t)|^{p-1}d\nu\, dt\to 0\mbox{   as   }h\to \infty.
\end{equation}

\item For all $v\in L^p((0,T); W^{s,p}(\O))\cap
L^{\infty}(\O_T)\cap C([0,T];L^1(\O))$ with $v_t\in
L^{p'}((0,T); W^{-s,p'}(\O))$ we have
\begin{equation}\label{eq:def1}
\begin{array}{lll}
& \dyle\io \Theta _k(u-v)(x,T)dx-\int_0^T\langle v_t,T_k(u-v)\rangle dt\\
&+ \dyle \frac 12\int_0^T\iint_{D_\O}
\,U(x,y,t)[T_k(u(x,t)-\varphi(x,t))-T_k(u(y,t)-\varphi(y,t))]d\nu\,dt \\
& \dyle\le \io \Theta _k(u_0(x)-v(x,0))dx+\iint_{\O_T}f T_k(u-v)dxdt,
\end{array}
\end{equation}
where $\Theta_k(\sigma)=\dyle\int_0^\sigma T_k(a)da.$
\end{enumerate}
\end{Definition}

 We will prove that for nonnegative data $(f,u_0)\in L^1(\O_T)\times L^1(\O)$,  the weak  solution obtained in the previous  Section, is an entropy  solution in the sense of Definition \ref{def1}. Notice that, as a by product, we recover the proof that \textit{any solution as limit of approximations is an entropy solution}, as in the local case.

\begin{Theorem}\label{entropy}
Assume that $(f,u_0)\in L^1(\O_T)\times L^1(\O)$ \, are nonnegative functions, then the weak solution to problem \eqref{eq:def} obtained in Theorem \ref{th1} is an entropy non negative solution in the sense of Definition \ref{def1}.
\end{Theorem}
\pf We have just to show that the weak solution obtained in Theorem \ref{th1} satisfies the conditions \eqref{entro001} and \eqref{eq:def1} stated in Definition \ref{def1}. It is clear that, in this case, the sequence $\{u_n\}_n$ of solution to the approximating problems \eqref{pro:lineal1} is increasing in $n$ and then $u_n\uparrow u$ a.e in $\O_T$.

Let us begin by proving estimate \eqref{entro001}. Since $u, u_n\ge 0$, then the set $R_h$ defined in \eqref{regi} is reduced to
$$
R_h=\bigg\{(x,y,t)\in \re^{2N}\times(0,T): h+1\le \max\{u(x,t),u(y,t)\}\mbox{  with  } \min\{u(x,t),u(y,t)\}\le h\bigg\}. $$
Using $T_1(G_h(u_n))$ as a test function in \eqref{pro:lineal1}, it follows that
\begin{equation*}
\begin{array}{lll}
&\dyle \iint_{\O_T}u_{nt} T_1(G_h(u_n(x,t)))\,dx\,dt\\
&+ \dyle \dfrac 12\int_0^T\dyle \iint_{D_\O}
\,U_n(x,y,t)[T_1(G_h(u_n(x,t)))-T_1(G_h(u_n(y,t)))]d\nu \,dt\\
&=\dyle \iint_{\O_T} f_n(x,t)T_1(G_h(u_n(x,t)))\, dx\,dt\le \iint_{\O_T\cap \{u_n\ge h\}}f_n(x,t)dx\,dt.
\end{array}
\end{equation*}
Notice that
$$
\iint_{\O_T}u_{nt} T_1(G_h(u_n(x,t)))\,dx\,dt=\dyle\io \tilde{\Theta}_h(u_n)(x,T)dx-\dyle\io \tilde{\Theta}_h(u_n)(x,0)dx
$$
where $\tilde{\Theta}_h(\sigma)=\dyle\int_0^\sigma T_1(G_h(a))da.$ It is clear that $\tilde{\Theta}_h(\sigma)\le \sigma$ for all $\sigma\ge 0$.

Taking into consideration that $u_n\ge 0$, it holds
\begin{equation*}
\begin{array}{lll}
&\dfrac 12\dyle \int_0^T \iint_{D_\O}
\,U_n(x,y,t)[T_1(G_h(u_n(x,t)))-T_1(G_h(u_n(y,t)))]d\nu \,dt\\
& \le \dyle\io \tilde{\Theta}_h(u_n)(x,0)dx+\iint_{\O_T\cap \{u_n\ge h\}}f_n(x,t)dx\,dt\\
& \le \dyle\int_{u_0>h}u_0(x)dx+\iint_{\O_T\cap \{u_n\ge h\}}f_n(x,t)dx\,dt.
\end{array}
\end{equation*}
It is not difficult to show that
$$
U_n(x,y,t)[T_1(G_h(u_n(x,t)))-T_1(G_h(u_n(y,t)))]\ge 0.
$$
Thus, using Fatou's lemma, we conclude that
\begin{eqnarray*}
&\dfrac 12\dyle \int_0^T\iint_{D_{\O}}
\,U(x,y,t)[T_1(G_h(u(x,t)))-T_1(G_h(u(y,t)))]d\nu\,dt\le \\
&\dyle \liminf_{n\to \infty} \dfrac 12\dyle \int_0^T\iint_{D_{\O}}
\,U_n(x,y,t)[T_1(G_h(u_n(x,t)))-T_1(G_h(u_n(y,t)))]d\nu\,dt\\
&\le \dyle\int_{u_0\ge h}u_0(x)dx+\iint_{\O_T\cap \{u_n\ge h\}}f_n(x,t)dx\,dt.
\end{eqnarray*}
Since
$$
U_n(x,y,t)[T_1(G_h(u(x,t)))-T_1(G_h(u(y,t)))]\ge |u(x,t)-u(y,t)|^{p-1}\: \mbox{ in } R_h,
$$
then, using the fact that
$$
\dyle\int_{u_0\ge h}u_0(x)dx+\iint_{\O_T\cap \{u_n\ge h\}}f_n(x,t)dx\,dt\to 0\mbox{   as   }h\to \infty,
$$
we conclude that
$$
\iiint_{R_h}
\,|u(x,t)-u(y,t)|^{p-1}d\nu\,dt\to 0\mbox{   as   }h\to \infty
$$
and then \eqref{entro001} holds.

Let now $v\in L^{p}(0,T; W^{s,{p}}_{ 0}(\Omega)) \cap L^{\infty}(\O_T)$ be such that $L^{p'}(0,T; W^{-s,{p'}}_{ 0}(\Omega))$. Taking $T_k(u_n-v)$ as a test function in \eqref{pro:lineal1}, we reach that
\begin{equation}\label{cont}
\begin{array}{lll}
&\dyle \iint_{\O_T}u_{nt} T_k(u_n-v)\,dx\,dt\\
& +\dyle \dfrac 12\dyle \int_0^T\iint_{D_{\O}}
\, U_n(x,y,t)[T_k(u_n(x,t)-v(x,t))-T_k(u_n(y,t)-v(y,t))]d\nu\,dt\\
&=\dyle \iint_{\O_T} f_n(x,t)T_k(u_n(x,t)-v(x,t)) \, dx\,dt.
\end{array}
\end{equation}
Let us study the limit, as $n\to \infty$, of each term of the pervious identity.

By the Dominated Convergence theorem one can easily show that, as $n\to \infty$,
$$
\dyle \iint_{\O_T}f_n(x,t)T_k(u_n(x,t)-v(x,t)) \, dx\,dt\to \dyle \iint_{\O_T} f(x,t)T_k(u(x,t)-v(x,t)) \, dx\,dt.
$$
Since $u_{nt}= (u_n-v)_t+v_t$ one has
\begin{eqnarray*}
&\dyle \iint_{\O_T}u_{nt} T_k(u_n-v)\,dx\,dt=\\
&\dyle \int_{\O}[\Theta_k(u_n-v)](T)\,dx
-\int_{\O}[\Theta_k(u_n-v)](0)dx-\iint_{\O_T}v_tT_k(u_n-v)\,dx\,dt.
\end{eqnarray*}
Using the fact that $u_n\to u$ strongly in $C([0,T], L^1(\O))$ and since $\Theta_k$ is Lipschitz continuous, one has, as $n\to\infty$,
$$ \int_{\O}[\Theta_k(u_n-v)](T)\,dx\to \int_{\O}[\Theta_k(u-v)](T)\,dx
$$
and $$
\int_{\O}[\Theta_k(u_n-v)](0)\,dx\to \int_{\O}[\Theta_k(u_0-v(0))]\,dx.$$
We analyze now the term $\dyle\iint_{\O_T}v_tT_k(u_n-v)\,dxdt$. Since $v\in L^{\infty}(\O_T)$, letting $M=||v||_{\infty}$,  then  $T_k(u_n-v)=T_k(T_{M+k}(u_n)-v)$. Thus $T_k(u_n-v)\rightharpoonup T_k(u-v)$ weakly in $L^p(0,T; W^{s,p}_{ 0}(\Omega)) $. As $v_t\in L^{p'}(0,T; W^{-s,{p'}}_{ 0}(\Omega))$, then a duality argument allows us to conclude that
$$\iint_{\O_T}v_tT_k(u_n-v)\,dx\,dt \to \iint_{\O_T}v_tT_k(u-v)\,dx\,dt.$$
We deal now with the second term in \eqref{cont}. We follow closely the same arguments as in \cite{AAB}, for the reader convenience and to make the paper self contained we include here all details.

We set
$$
w_n=u_n-v\mbox{  and } W_n(x,y,t)=|w_n(x,t)-w_n(y,t)|^{p-2}(w_n(x,t)-w_n(y,t)),$$
then
$$
U_n(x,y,t)[T_k(u_n(x,t)-v(x,t))-T_k(u_n(y,t)-v(y,t))]=:K_{1,n}(x,y,t)+K_{2,n}(x,y,t)
,$$
where
\begin{eqnarray*}
K_{1,n}(x,y,t)=W_n(x,y,t)[T_k(w_n(x,t))-T_k(w_n(y,t))],
\end{eqnarray*}
and
\begin{eqnarray*}
K_{2,n}(x,y,t)=\Big[U_n(x,y,t)-W_n(x,y,t)\Big][T_k(w_n(x,t))-T_k(w_n(y,t))].
\end{eqnarray*}
It is clear that $K_{1,n}(x,y,t)\ge 0$ \ a.e. in $D_{\Omega}\times (0,T)$, since
\begin{eqnarray*}
K_{1,n}(x,y,t)&\to & W(x,y,t)[T_k(w(x,t))-T_k(w(y,t))] \  a.e. \mbox{  in   }D_{\Omega_T},
\end{eqnarray*}
as $n\to \infty$, where
$$
w=u-v\mbox{  and } W(x,y,t)=|w(x,t)-w(y,t)|^{p-2}(w(x,t)-w(y,t)).$$
Hence, using Fatou's Lemma, we obtain that
\begin{eqnarray*}
&\dyle \int_0^T\iint_{D_{\O}}K_{1,n}(x,y,t)d\nu\, dt\ge\\
&\dyle \int_0^T\iint_{D_{\O}}W(x,y,t)[T_k(w(x,t))-T_k(w(y,t))]  d\nu\, dt.
\end{eqnarray*}
We deal now with $K_{2,n}$.

We set
$$
\s_1(x,y,t)=u_n(x,t)-u_n(y,t)\mbox{  and  }\s_2(x,y,t)=w_n(x,t)-w_n(y,t).$$ Then
$$
K_{2,n}(x,y,t)=\Big[|\s_1(x,y,t)|^{p-2}\s_1(x,y,t)-|\s_2(x,y,t)|^{p-2}\s_2(x,y,t)\Big]$$
$$
\times [T_k(w_n(x,t))-T_k(w_n(y,t))].
$$
We claim that, as $n\to \infty$,
\begin{equation}\label{cll}
\begin{array}{lll}
&\dyle \int_0^T\iint_{D_{\O}} K_{2,n}(x,y,t)d\nu\,dt\to \\
& \dyle \int_0^T\iint_{D_{\O}}\Big[U(x,y,t)-W(x,y,t)\Big][T_k(w(x,t))-T_k(w(y,t))]d\nu\,dt
\end{array}
\end{equation}
We divide the proof of the claim into two cases according to the value of $p$.

\

\noindent {\bf \emph{The singular case $p\in (1,2]$:}}
In this case we have
\begin{eqnarray*}
& \Big||\s_1(x,y,t)|^{p-2}\s_1(x,y,t)-|\s_2(x,y,t)|^{p-2}\s_2(x,y,t)\Big| \\
& \le C|\s_1(x,y,t)-\s_2(x,y,t)|^{p-1}=C|v(x,t)-v(y,t)|^{p-1}.
\end{eqnarray*}
Thus
$$
|K_{2,n}(x,y,t)|\le C|v(x,t)-v(y,t)|^{p-1}|T_k(w(x,t))-T_k(w(y,t))|\equiv \tilde{K}_{2,n}(x,y,t).
$$
Using Lemma \ref{lm:l1}, we get that
$$|T_k(w_n(x,t))-T_k(w_n(y,t))|\to |T_k(w(x,t))-T_k(w(y,t))| \mbox{  strongly  in  }L^{p'}(D_{\O_T},d\nu).$$
Since $v\in  L^{p}(0,T; W^{s,{p}}_{ 0}(\Omega)) \cap L^{\infty}(\O_T)$, by duality argument we conclude that
$$
\tilde{K}_{2,n}\to C|v(x,t)-v(y,t)|^{p-1}|T_k(w(x,t))-T_k(w(y,t))|\mbox{  strongly  in  }L^1(D_{\O_T},d\nu).
$$
Using the Dominated Convergence theorem we reach that
\begin{eqnarray*}
&\dyle \int_0^T\iint_{D_{\O}}K_{2,n}(x,y,t)d\nu\,dt\to \\
& \dyle \int_0^T\iint_{D_{\O}}\Big[U(x,y,t)-W(x,y,t)\Big][T_k(w(x,t))-T_k(w(y,t))]d\nu\,dt,
\end{eqnarray*}
as $n\to \infty$ and the claim follows in this case.

\

\noindent {\bf{\emph{The degenerate case $p>2$:}}} In this case we have
\begin{eqnarray*}
& \Big||\s_1(x,y,t)|^{p-2}\s_1(x,y,t)-|\s_2(x,y,t)|^{p-2}\s_2(x,y,t)\Big|\\
&\le C_1|\s_1(x,y,t)-\s_2(x,y,t)|^{p-1}+C_2|\s_2(x,y,t)|^{p-2}|\s_1(x,y,t)-\s_2(x,y,t)|\\
&\le C_1|v(x,t)-v(y,t)|^{p-1} +C_2|v(x)-v(y)||w_n(x,t)-w_n(y,t)|^{p-2}\\
&\le C_1|v(x,t)-v(y,t)|^{p-1} +C_2|v(x,t)-v(y,t)||u_n(x,t)-u_n(y,t)|^{p-2}.
\end{eqnarray*}
Thus
\begin{eqnarray*}
|K_{2,n}(x,y,t)|& \le & C_1|v(x,t)-v(y,t)|^{p-1}|T_k(w_n(x,t))-T_k(w_n(y,t))|\\
&+ & C_2|v(x,t)-v(y,t)||u_n(x,t)-u_n(y,t)|^{p-2}|T_k(w_n(x,t))-T_k(w_n(y,t))|\\
&\equiv & \bar{K}_{2,n}(x,y,t)+\check{K}_{2,n}(x,y,t).
\end{eqnarray*}
The term $\bar{K}_{2,n}(x,y,t)$ can be treated as $\tilde{K}_{2,n}$ above. Hence it remains to deal with $\check{K}_{2,n}(x,y,t)$.

We define
$$D_1=\{(x,y,t)\in D_{\O_T}: u_n(x,t)\le \tilde{k}, u_n(y,t) \le \tilde{k} \},$$ where $\tilde{k}>>k+||v||_\infty$ is a large constant. Using duality argument we obtain
\begin{eqnarray*}
&\check{K}_{2,n}(x,y,t)\chi_{D_1}\to \\
& C_2|v(x,t)-v(y,t)||u(x,t)-u(y,t)|^{p-2}|T_k(w(x,t))-T_k(w(y,t))|\chi_{\{u(x,t)\le \tilde{k}, u(y,t)\le \tilde{k}\}}
\end{eqnarray*}
strongly  in  $L^1(D_{\O_T},d\nu)$.

\

Now, consider the set
$$D_2=\{(x,y,t)\in D_{\O_T}: u_n(x,t)\ge k_1, u_n(y,t) \ge k_1 \},$$ where $k_1>k+||v||_\infty$, then $\check{K}_{2,n}(x,y,t)\chi_{D_2}(x,y,t)=0$.

It is clear that, taking into consideration the previous computations, that we have just to analyze the convergence on the set
$$D_3=\{(x,y,t)\in D_{\O_T}: u_n(x,t)\ge 2k, u_n(y,t) \le k \},$$ or
$$D_4=\{(x,y,t)\in D_{\O_T}: u_n(y,t)\ge 2k, u_n(x,t) \le k \}.$$

If $(x,y,t)\in D_3$, then
\begin{eqnarray*}
& \check{K}_{2,n}(x,y,t)\chi_{D_3}(x,y,t) \\
& \le C(k)\bigg|v(x,t)-v(y,t)\bigg|\bigg|T_k(w_n(x,t))-T_k(w_n(y,t))\bigg|u^{p-2}_n(x,t)\chi_{D_3}(x,y,t).
\end{eqnarray*}
Notice that
$$
u^{p-2}_n(x,t)\chi_{D_3}(x,y,t)\rightharpoonup u^{p-2}(x,t)\chi_{\{u(x,t)\ge 2k, u(y,t)\le k\}} \mbox{  weakly  in   }L^{\frac{p-1}{p-2}}(D_{\O_T}, d\nu).
$$
Since
\begin{eqnarray*}
&\Big[|v(x,t)-v(y,t)||T_k(w_n(x,t))-T_k(w_n(y,t))|\Big]^{p-1}\chi_{\{u(x,t)\ge 2k, u(y,t)\le k\}} \\
&\le k^{p-2}|v(x,t)-v(y,t)|^{p-1}\bigg|T_k(w_n(x,t))-T_k(w_n(y,t))\bigg|\chi_{\{u(x,t)\ge 2k, u(y,t)\le k\}}
\end{eqnarray*}
The duality argument allows us to conclude that
\begin{eqnarray*}
&\Big[|v(x,t)-v(y,t)||T_k(w_n(x,t))-T_k(w_n(y,t))|\Big]^{p-1}\chi_{\{u(x,t)\ge 2k, u(y,t)\le k\}} \to \\ &\Big[|v(x,t)-v(y,t)||T_k(w(x,t))-T_k(w(y,t))|\Big]^{p-1}\chi_{\{u(x,t)\ge 2k, u(y,t)\le k\}}
\end{eqnarray*}
strongly   in   $L^1(D_\O, d\nu)$ as $n\to \infty$.

Thus
\begin{eqnarray*}
& \check{K}_{2,n}\chi_{D_3}\to \\
& C_2|v(x,t)-v(y,t)||u(x,t)-u(y,t)|^{p-2}|T_k(w(x,t))-T_k(w(y,t))|\chi_{\{u(x,t)\ge 2k, u(y,t)\le k\}}
\end{eqnarray*}
strongly  in $L^1(D_{\O_T},d\nu)$.

\

In the same way we can treat the set $D_4$.

Therefore, combining the above estimates and using the Dominate Convergence theorem, we conclude that
\begin{eqnarray*}
&\dyle \int_0^T\iint_{D_{\O}}K_{2,n}(x,y,t)d\nu\,dt\to \\
& \dyle \int_0^T\iint_{D_{\O}}\Big[U(x,y,t)-W(x,y,t)\Big][T_k(w(x,t))-T_k(w(y,t))]d\nu\,dt
\end{eqnarray*}
as $n\to \infty$ and the claim follows in this case.

Therefore, as a conclusion we have proved that
\begin{equation*}
\begin{array}{lll}
&\dyle \int_{\O}[\Theta_k(u-v)](T)\,dx
-\int_{\O}[\Theta_k(u-v)](0)dx-\iint_{\O_T}v_tT_k(u-v)\,dx\,dt\\
&+ \dfrac 12\dyle \int_0^T\iint_{D_{\O}}
\, U(x,y,t)[T_k(u(x,t)-v(x,t))-T_k(u(y,t)-v(y,t))]d\nu\,dt\\
&\le\dyle \iint_{\O_T} f(x,t)T_k(u(x,t)-v(x,t)) \, dx\,dt.
\end{array}
\end{equation*}
and the result follows at once. \cqd

\section{Further results.}
\subsection{Extinction in the finite time}

In this subsection we suppose that $f\equiv 0$ and $p<2$, our main goal is to get natural condition in order to show that the nonnegative solution is zero for large time. The first result in this direction is the following.

\begin{Theorem}\label{exis2}
Assume that $\frac{2N}{N+2s}\le p<2$ and $u_0\in L^2(\O)$. Let $u$ be the unique nonnegative solution to the problem .
\begin{equation}\label{propa}
\left\{
\begin{array}{rcll}
u_t+(-\D^s_{p}) u &= & 0  &
\text{ in } \O_{T}, \\
u & \ge &  0 &
\text{ in }\Omega, \\
u &= & 0  & \text{ in }\ren\setminus\O \times (0,T), \\
u(x,0) & = & u_0(x) & \mbox{  in } \O,
\end{array}%
\right.
\end{equation}
then there exists a finite time $T^*(N,p,|\O|,||u_0||_2)\equiv T^*\le \frac{2}{(2-p)S}||u_0||^{2-p}_2 |\O|^{\frac{p}{2}-\frac{p}{p^*_s}}$ such that
$u(.,t)\equiv 0$ for $t\ge T^*$.
\end{Theorem}

\begin{proof}
We follow closely the arguments used in \cite{AP}. For the reader convenience we include here some details. Using $u$ as a test function in \eqref{propa}, we get
\begin{equation*}
\frac{1}{2}\frac{d}{dt}\int_{\Omega }u^{2}dx+ \frac{1}{2}\iint_{D_{\O}}\dfrac{|u(x,t)-u(y,t)|^{p}}{|x-y|^{N+ps}}dx \ dy= 0.
\end{equation*}
By the Sobolev inequality, we reach that
\begin{equation*}
\frac{1}{2}\frac{d}{dt}\int_{\Omega }u^{2}dx+\frac{S}{2} \left( \int_{\Omega }\left\vert u\right\vert ^{p^{\ast }_s}dx\right) ^{\frac{p}{p^\ast_s }}\leq 0.
\end{equation*}
Suppose that $\frac{2N}{N+2s}<p<2$, then $p^*_s>2$, thus by H\"older inequality, we obtain
$$\int_{\O}u^2(x,t)dx\le |\O|^{1-\frac{2}{p^*_s}}\big(\int_{\O}\mid u^{p^*_s}(x,t)\ dx\big)^{\frac{2}{p^*_s}}.$$
Thus
$$\frac{1}{2}\frac{d}{dt}\parallel u(x,t)\parallel^2_2+\frac{S}{2} |\O|^{\frac{p}{p^*_s}-\frac{p}{2}}\parallel u(x,t)\parallel^p_2\le 0
$$
and then
$$\parallel u(x,T)\parallel_2\le \parallel u_0\parallel_2\Big(1- \frac{(2-p)\frac{S}{2}|\O|^{\frac{p}{p^*_s}-\frac{p}{2}}T}{\parallel u_0\parallel^{2-p}_2}\Big)^{\frac{1}{2-p}}.
$$
Hence if $T\ge T^*\equiv \frac{2}{(2-p)S}||u_0||^{2-p}_2 |\O|^{\frac{p}{2}-\frac{p}{p^*}}$, then $u(x,T)=0$ and the result follows.
\end{proof}

In the case where $1<p<\frac{2N}{N+2s}$, under suitable hypothesis on $u_0$, we can prove the finite time extinction property. More precisely we have.
\begin{Theorem}
Suppose that $1<p<\frac{2N}{N+2s}$ and $u_0 \in L^{\nu+1}(\O)\cap  L^{2}(\O)$ with $\nu+1=\frac{N(2-p)}{ps}$, then there exists $T^*$ such that $u( .,t) \equiv 0$ for all $t\geq T^*$.
\end{Theorem}
\begin{proof}
We use $u^{\nu}$ as test function in \eqref{propa}, then
\begin{eqnarray*}
& \dyle \frac{1}{\nu+1}\frac{d}{dt}\int_{\Omega }u^{\nu+1}dx\\
&\dyle +\frac{1}{2}\iint_{D_{\O}}\dfrac{|u(x,t)-u(y,t)|^{p-2}(u(x,t)-u(y,t))}{|x-y|^{N+ps}}(u^\nu(x,t)-u^\nu(y,t))dx \ dy
=0.
\end{eqnarray*}
Hence, by inequality\eqref{alge3}, we get
\begin{equation*}
\frac{1}{\nu+1}\frac{d}{dt}\int_{\Omega }u^{\nu+1}dx+\frac{C}{2}\iint_{D_{\O}}\dfrac{|u^{\frac{p+\nu-1}{p}}(x,t)-u^{\frac{p+\nu-1}{p}}(y,t)|^{p}}{|x-y|^{N+ps}}dy \ dx\le 0
\end{equation*}%
Using now Sobolev inequality there results
\begin{equation*}
\frac{1}{\nu+1}\frac{d}{dt}\int_{\Omega }u^{\nu+1}dx+C\left( \int_{\Omega }u^{\frac{\left( \nu+p-1\right) }{p}p^{\ast }_s}dx\right) ^{\frac{p}{p^{\ast }_s}}\leq 0.
\end{equation*}
Recall that $\nu=\frac{N(2-p)-ps}{ps}$, then  $\frac{\nu+p-1}{p}p^{\ast }_s=\nu+1 $.\\

$$\frac{1}{\nu+1}\frac{d}{dt} ||u(x,t)||^{\nu+1}_{\nu+1}+S||u(x,t)||^{\nu+p-1}_{\nu+1} \le 0$$
Now,  we get that
$$
||u(x,T)||_{\nu+1}\le ||u_0||_{\nu+1}\Big(1-\frac{c(2-p)T}{\parallel u_0\parallel^{2-p}_{\nu+1}}\Big)^{\frac{1}{2-p}}.
$$
Hence the result follows.
\end{proof}
\subsection{Non Finite speed of propagation} It is wellknown that for the local $p$-laplacian parabolic problem with $p>2$, there is a phenomenon of \textit{finite speed of propagation}. In fact,  the fundamental solution obtaided by G. Barenblatt  allows to prove  finite speed of propagation by using comparison arguments.

The meaning of finite speed of propagation in the local case can be summarized as follows:

\textit{Assume that we have an inial data such that
 $\text{supp}(u_0)$ is a compact set, then $\text{supp}(u(,t))$ is a compact set of $\O$ for $t<t_1$. }

We can rewrite the previous notion by saying that:

\textit{ Given an initial data with finite support, $u_0$, for all $t>0$, there exists $R>0$ such that $u(x,t)=0$ if $|x|>R$.}

Let us consider the nonlocal problem \eqref{propa} with $\Omega\equiv \ren$ and a bounded nonnegative data $u_0$ with compact support.

If we assume that the finite speed of propagation holds, we get a contradiction with the fact that $u\in \mathcal{C}([0, T), L^1(\mathbb{R}^N))$.

Indeed, suppose that for $t_0>0$, there exists $x_0\in \mathbb{R}^N$ such that the solution verifies that $u(x_0,t_0)=0$. Then $(x_0,t_0)$ is a  global minimum, hence
$$0=\int_{\mathbb{R}^N} \frac{|u(y,t_0)|^{p-2}u(y,t_0)}{|x-y|^{N+ps}}dy.$$
Since $u(x,t)\ge 0$, we find that  $u(x,t_0)= 0$ for all $x\in\mathbb{R}^N$.

Recall that $u\in \mathcal{C}([0, T), L^1(\mathbb{R}^N))$, thus by continuity for $t$ small, we have
$ \dint_{\mathbb{R}^N} u(x, t)dx>0$ and then we reach a contradiction.

Notice that if $p>2$,  as in the local case,  by  a scaling arguments the equation can be reduced to the self-similar variable and the corresponding Barenblatt type solution can be obtained.  Following the radial computations  in \cite{FV} and \cite{G}, we get that a self-similar solution $u(x,t)=t^{-N\beta}\Upsilon(\frac{r}{t^\beta})$
with $r=|x|, \beta=\frac{1}{ps+N(p-2)}$, must to solve the following equation
\begin{equation*}\label{hh}
\b \big[ N\Upsilon(r)+r\Upsilon'(r)\big]=\frac{1}{r^{ps}}\int_0^\infty\,|\Upsilon(r)-\Upsilon(\sigma r)|^{p-2}(\Upsilon(r)-\Upsilon(\sigma r))\sigma^{N-1}K_s(\sigma)d\sigma
\end{equation*}
where
$$
K_s(\sigma)=\dint\limits_{|y'|=1}\dfrac{dH^{n-1}(y')}{|x'-\s
y'|^{N+ps}}=2\frac{\pi^{\frac{N-1}{2}}}{\beta(\frac{N-1}{2})}\int_0^\pi
\frac{\sin^{N-2}(\xi)}{(1-2\sigma \cos
(\xi)+\sigma^2)^{\frac{N+ps}{2}}}d\xi.
$$
Since $\Upsilon\gneqq 0$, then $\Upsilon(\s)> 0$ for all $\s>0$. We refer to \cite{VAZ} where additional properties of the previous profile and the asymptotic behavior are studied.
\subsection{Extinction for Concave case.}
Let consider now the problem
\begin{equation}\label{concave}
\left\{
\begin{array}{rcll}
u_t+(-\D^s_{p}) u & = &  u^q  &
\text{ in } \O_{T},\\
u & \ge &  0 &
\text{ in }\Omega, \\
u &=& 0 & \text{ in }\ren\setminus\O \times (0,T), \\
u(x,0) & = & u_0(x) & \mbox{  in  }\O,
\end{array}%
\right.
\end{equation}
where $q\le 1$, the result obtained is similar as in the \cite{AMTP} and \cite{MMP} in the local case. For the reader convenience we include the calculations in the fractional case.
\begin{Theorem}\label{concave-extinction}
Let  $1<p<2$ and $u_0\in L^2(\O)$, then for $u_0\in L^2(\O)$ and for all $p-1<q\le 1$ the problem \eqref{concave} has a nonnegative  minimal solution $u\in L^p(0,T; W^{s,p}_{ 0}(\Omega))$, moreover if $p>\frac{2N}{N+2s}$; then under a smallness condition
  on $||u_0||_2$, there exists a finite time  $T^*$ such that
$u(.,t)\equiv 0$ for all $t\ge T^*$.
\end{Theorem}
\begin{proof}
We begin by the case $q=1$. Let $u_n$ be the minimal solution of the approximated problem
\begin{equation}\label{pro:concave}
\left\{\begin{array}{rcll}
u_{nt}+(-\D^s_{p}) u_n &= & u^q_n & \mbox{  in   }\O_T,\\
u_n & = & 0 & \hbox{  in } \ren\backslash\O\times (0,T),\\
u_n(x,0) & = & u_{0n}(x) & \mbox{  in   }\O,
\end{array}
\right.
\end{equation}
taking $u_n$ as a test function in \eqref{pro:concave}, we obtain
\begin{equation}\label{pro:concave 1}
\frac{1}{2}\frac{d}{dt}\int_{\Omega }u^{2}_ndx- \int_{\Omega }u^{2}_ndx+ \frac{1}{2}\iint_{D_{\O}}\dfrac{|u_n(x,t)-u_n(y,t)|^{p}}{|x-y|^{N+ps}}dxdy= 0.\end{equation}
Thus by Gronwall inequality we conclude
$$ \int_{\Omega }u^{2}_n(x,T)dx+ \frac{1}{2}\int_0^T\iint_{D_{\O}}\dfrac{|u_n(x,t)-u_n(y,t)|^{p}}{|x-y|^{N+ps}}dxdy\le ||u_0||^2_{2}e^{2 T}.$$
Therefore we reach that $\{u_n\}_n$ is bounded in $ L^p(0; T;W^{s,p}_{ 0}(\Omega))\cap L^2(\O_T)$. Thus $u_n\uparrow u$ with $u\in L^p(0,T; W^{s,p}_{ 0}(\Omega))$  and $u$ is the minimal solution to problem \eqref{concave}.
\newline
Let us assume now that $\frac{2N}{N+2s}<p<2$, using Sobolev inequality in \eqref{pro:concave 1} there result that
\begin{equation*}
\frac{1}{2}\frac{d}{dt}\int_{\Omega }u^{2}_ndx- \int_{\Omega }u^{2}_n dx+\frac{S}{2} \left( \int_{\Omega }\left\vert u_n\right\vert ^{p^{\ast }_s}dx\right) ^{\frac{p}{p^\ast_s }}\leq 0.
\end{equation*}
Thus
\begin{equation*}
\frac{d}{dt}\Big(e^{-2t}\int_{\Omega }u^{2}_ndx\Big)+S e^{-(2-p)t} \left( \int_{\Omega }\left(e^{-2t}\vert u_n\right\vert) ^{p^{\ast }_s}dx\right)^{\frac{p}{p^\ast_s }}\leq 0.
\end{equation*}
Since $\frac{2N}{N+2s}<p<2$, then $p^*>2$, therefore by setting $F(t)=e^{-2t}\dyle\int_{\Omega }u^{2}_ndx$ and using H\"older inequality, it follows that
\begin{equation*}
\dfrac{F'(t)}{F^{\frac p2}(t)}\leq -C e^{-(2-p)t}.
\end{equation*}
Integrating in time, we obtain that
\begin{equation*}
F^{1-\frac{p}{2}}(t)\leq F^{1-\frac{p}{2}}(0)+C \big[\frac{1}{2-p} e^{-(2-p)t}- \frac{1}{2-p}\big].
\end{equation*}
Thus
\begin{equation*}
F(t)\leq\Big[ F^{1-\frac{p}{2}}(0)+C \big(\frac{1}{2-p} e^{-(2-p)t}- \frac{1}{2-p}\big)\Big]^{\frac{2}{2-p}}.
\end{equation*}
Recalling that $F\left( 0\right) =\int_{\Omega }u^2\left( x,0\right)dx=\left\Vert u_{0}\right\Vert^2_{L^{2}};$
so if \begin{equation*}
\left\Vert u_{0}\right\Vert _{L^{2}}\leq \left[\frac{C}{2-p}\right] ^{^{\frac{1}{2-p}}}.
\end{equation*}%

We obtain $F(t)\le 0$ for some $T^{\ast }=T^{\ast }\left(C,p\right)$ and then the extinction result follows.

\noindent Let consider now the case where $q<1$. It is not difficult to see that the same estimates as above allow us to get the existence of minimal solution.
Hence we have just to proof the extinction result.

Since $p>\frac{2N}{N+2s}$, then we get the existence of $0<\nu<1$, closed to $1$ such that
\begin{equation*}
\left( \nu+p-1\right) \frac{N}{N-p}>\nu+1>q+\nu.
\end{equation*}
Using $(u_n+\e)^{\nu}-\e^{\nu}$, $\nu>0;\ $ as test function in \eqref{pro:concave}, if holds
\begin{equation*}
\begin{array}{lll}
& \dyle \frac{d}{dt}\int_{\Omega }\frac{(u_n+\e)^{\nu+1}}{\nu+1}dx + \dyle\frac{1}{2}\iint_{D_{\O}}\dfrac{U_n(x,y,t)((u_n+\e)^\nu(x,t)-(u_n+\e)^\nu(y,t))}{|x-y|^{N+ps}}dx \ dy\\
& \leq \dyle \int_{\Omega }(u_n+\e)^{q+\nu
}dx.
\end{array}
\end{equation*}

Hence, by inequality\eqref{alge3}, we get
\begin{equation*}\label{TR10}
\begin{array}{lll}
&\dyle\frac{C}{2}\iint_{D_{\O}}\dfrac{|(u_n+\e)^{\frac{p+\nu-1}{p}}(x,t)-(u_n+\e)^{\frac{p+\nu-1}{p}}(y,t)|^{p}}{|x-y|^{N+ps}}dx \ dy\\
&+\dyle \frac{1}{\nu+1}\frac{d}{dt}\int_{\Omega }(u_n+\e)^{\nu+1}dx\leq \int_{\Omega }(u_n+\e)^{q+\nu}dx.
\end{array}
\end{equation*}
Using now Sobolev inequality, it follows that
\begin{equation}\label{GT}
\frac{1}{\nu+1}\frac{d}{dt}\int_{\Omega }(u_n+\e)^{\nu+1}dx+C\left( \int_{\Omega }(u_n+\e)^{\frac{\left( \nu+p-1\right) }{p}p^{\ast }_s}dx\right) ^{\frac{p}{p^{\ast }_s}}\leq \int_{\Omega }(u_n+\e)^{q+\nu}dx.
\end{equation}
Since $\left( \nu+p-1\right) \frac{N}{N-ps}>\nu+1>q+\nu$, then using Young inequality, there results that
\begin{eqnarray*}
\int_{\Omega }(u_n+\e)^{q+\nu}dx &\leq &C_{\eta}\int_{\Omega }(u_n+\e)^{\nu+1}dx+\eta\int_{\Omega
}(u_n+\e)^{\nu+p-1}dx \\
&\leq &C_{\eta}\int_{\Omega }(u_n+\e)^{\nu+1}dx+\eta C(\O)\left( \int_{\Omega }(u_n+\e)^{\frac{\nu+p-1}{p}p^{\ast }_s}dx\right)^{\frac{p}{p^{\ast }_s}}.
\end{eqnarray*}
By substituting in \eqref{GT} and choosing $\eta$ small enough, we conclude that
\begin{equation*}
\frac{1}{\nu+1}\frac{d}{dt}\int_{\Omega }(u_n+\e)^{\nu+1}dx+C_1\left( \int_{\Omega }(u_n+\e)^{\frac{\left( \nu+p-1\right) }{p}p^{\ast }_s}dx\right) ^{\frac{p}{p^{\ast }_s}}\leq C_2 \int_{\Omega }(u_n+\e)^{1+\nu}dx,
\end{equation*}
with $C_1,C_2>0$ depending only on the data and are independent of $n$ and $\e$.

Passing to the limit as $\e \to 0$, and by setting $F(t) =\dyle e^{-C_2t}\int_{\Omega }u^{\nu+1}_n dx,$ as in the case $q=1$,
\begin{equation*}
F^{\prime }+C_3e^{C_4t}F^{\frac{\nu+p-1}{\nu+1}}\leq 0,
\end{equation*}%

where $C_3>0$ depends only on $\nu,N,p, \O$ and $C_4=\frac{\nu+p+1}{\nu+1}C_2$. Thus as in the first case, if $F(0)\equiv \int_{\Omega }u_0^{\nu+1}dx$ is small then we get a finite time
extinction. Since $\O$ is a bounded domain, then under the condition that $||u_0||_{L^2}$ is small we get the same conclusion. Hence the proof is complete.
\end{proof}

In the case where $1<p<\frac{2N}{N+2s}$, under suitable hypothesis on $u_0$, we can prove the finite time extinction property. More precisely we have.
\begin{Theorem}
Suppose that $1<p<\frac{2N}{N+2s}$, $p-1<q\leq 1$ and $u_0 \in L^{\nu+1}(\O)\cap  L^{2}(\O)$ with $\nu+1=\frac{N(2-p)}{ps}$, then there exists $T^*$ such that $u( .,t) \equiv 0$ for all $t\geq T^*$.
\end{Theorem}

\

If $q<p-1$, then a different phenomenon appears, more precisely we have the following result.

\begin{Theorem}\label{nonexis2} Assume that $1<p<2$ and let $q<p-1$, then the problem
\begin{equation}\label{non1}
\left\{\begin{array}{rcll} u_t + (-\D^s_{p})u&=& u^{q} & \mbox{ in } \O\times (0,T),\\
 u&=&0 &\hbox{  in
\  } \ren\backslash\O\times (0,T),\\
u(x,0)&=& 0 & \mbox{ in }\O,
\end{array}
\right.
\end{equation}
has a global solution $u$ such that $u(x,t)>0$ for all $t>0$ and
$x\in \O$, namely there is non finite time extinction, moreover, $u(.,t)\uparrow w$ as $t \to \infty$ where
$w$ is the unique positive solution to problem
\begin{equation}\label{elnon1}
\left\{\begin{array}{rcll} (-\D^s_{p})w&=& w^q & \mbox{ in } \O,  \\w&=&0 & \hbox{  in
\  } \ren\backslash\O.
\end{array}
\right.
\end{equation}
\end{Theorem}
\begin{proof}
The proof use the sub-supersolution argument. Without loss of
generality we can assume that $\O\subset B_1(0)$. Since $p<2$, then
$q<1$.

Let us begin by the construction of a suitable subsolution.

Define $\mu(t)=(1-q)t^\frac{1}{1-q}$, it is clear that $\mu$ solves $\mu'=c\mu^q$
with $\mu(0)=0$. Consider $w$ the unique positive solution to problem \eqref{elnon1}, then $w\in L^\infty(\O)$. Setting $V(x,t)=\mu(\e t)w(x)$,
it holds that $v$ solves
\begin{eqnarray*}
V_t-\D^s _pV&=& \e\mu^q(\e t)w(x)+\mu^{p-1}(\e t)w^q \\
&\le &\e\mu^q(\e t)w(x)+\mu^{p-1}(\e
t)w^q(x).
\end{eqnarray*}
It is clear that $\mu^{p-1}(\e t)\le c_0\mu^{q}(\e t)$ for all $t\in
[0,T]$ and $w\le c_1w^q$ in $\O$, thus we can choose $\e$ and $C$,
depending only on $T$ and $c_1$ such that
$$
V_t-\D _pV\le V^q\mbox{  in   }\O\times
(0,T).
$$
Hence $v$ is a subsolution to problem \eqref{non1}. Using the fact
that $w$ is a supersolution to \eqref{non1} with $v\le cw$ in
$\O\times (0,T)$, then by the sub-supersolution argument we reach
the existence of a global solution $u\ge V$ in $\O\times (0,T)$. It
is clear that $u(x,t)>0$ for all $(x,t)\in \O\times (0,T)$.

To prove that  $u(.,t)\uparrow w$ as $t \to \infty$, for $t_0>0$, define $v(x,t,t_0)=u(x,t-t_0)$ where $t\ge t_0$ then $v$ solves:

 \begin{equation}\label{non11}
\left\{\begin{array}{rcll} v_t-\D^s_pv&=& v^{q}& \mbox{ in } \O\times (t_0,T),\\
 v&=&0 & \hbox{  in
\  } \ren\backslash\O\times (t_0,T),\\
v(x,t_0)&=& 0 & \mbox{ in }\O.
\end{array}
\right.
\end{equation}
Since $u(x,t_0)>0$, by the weak comparison principle we reach that $v\leq u$ for all $t>t_0$. As $t_0$ is arbitrary, then $u$ cannot converge to 0.
 Now using $u_n$ as a test function in \eqref{concave-extinction}, it follows that

 \begin{equation*}
\frac{1}{2}\frac{d}{dt}\int_{\Omega }u^{2}_ndx+ \frac{1}{2}\iint_{D_{\O}}\dfrac{|u_n(x,t)-u_n(y,t)|^{p}}{|x-y|^{N+ps}}dx \ dy\le\int_{\Omega }u^{q+1}_n dx \end{equation*}
Since $q<p-1$, by H\"older and Poincar\'e inequalities, we reach that $$\parallel u_n(.,t)\parallel_{W^{s,p}_0(\O)}\le C. $$
uniformly in $t$. Moreover $u(x,t)\le w(x)$. Classical results implies that $ \lim\limits_{t\to \infty}u(x,t)=\bar{u}(x) $ exists and solves
\eqref{elnon1}. Hence by uniqueness we conclude that $\bar{u}=w$ and the results follows .
\end{proof}


\begin{thebibliography}{lio}
\bibitem{AAB} \textsc{B. Abdellaoui,  A. Attar, R. Bentifour,}   \emph{On  the Fractional p-laplacian equations with weight and general datum.} Adv. Nonlinear Anal. 2016.

\bibitem{AMPP} {\sc B. Abdellaoui, M. Medina, I. Peral, A. Primo},
{\em Optimal results for the fractional heat equation involving
the Hardy potential }. Nonlinear Anal. 140 (2016), 166-207.
\bibitem{AMTP} {\sc B. Abdellaoui, S.E. Miri, I. Peral,T.M. Touaoula},
{\em Some remarks on quasilinear parabolic
problems with singular potential and a reaction term}. Nonlinear Differ. Equ. Appl. 21 (2014), 453-490.

\bibitem{Adams} \textsc{R. A. Adams} \emph{Sobolev spaces}, Academic Press, New York, 1975.

\bibitem{AP} \textsc{J.A. Aguilar, I. Peral}, \emph{Global behavior of the Cauchy problem for some critical nonlinear parabolic equations}, SIAM J. Math. Anal Vol.31, No.6, (1270-1294)- 2000

\bibitem{BM}  \textsc{D. Blanchard, F. Murat}, \emph{Renormalised solutions of
nonlinear parabolic problems with $L\sp 1$ data: existence and
uniqueness}. Proc. Roy. Soc. Edinburgh Sect. A {\bf 127}, (1997),
no. 6, 1137-1152.

\bibitem{BMR}\textsc{ D. Blanchard, F. Murat, H. Redwane,} {\it Existence and Uniqueness of a Renormalized Solution for a Fairly General Class of Nonlinear Parabolic Problems}, Journal of Differential Equations Vol.177, 2, 2001, 331-374.

\bibitem{DaA}  {\sc A. Dall'Aglio}, {\em Approximated solutions of equations with $L\sp 1$ data. Application to the $H$-convergence of quasi-linear parabolic equations},  Ann. Mat. Pura Appl.,  { 170}  (1996), 207--240.

\bibitem{CKP} {\sc A. Di Castro, T. Kuusi, G. Palatucci}, {\em  Nonlocal Harnack
inequalities},  J. Funct. Anal. {\bf 267} (2014), no. 6, 1807-1836.

\bibitem{DPV} {\sc E. Di Nezza, G. Palatucci,  E. Valdinoci}, {\em Hitchhiker's guide to the fractional Sobolev
    spaces}, Bull. Sci. math. {\bf 136} (2012), no. 5, 521-573.


\bibitem{FV} \textsc{F. Ferrari, I. Verbitsky}, {\it Radial fractional Laplace operators and Hessian
inequalities}, J. Differential Equations {\bf 253}, (2012), no. 1,
244-272.

\bibitem{G} \textsc{L. Grafakos}, {\it Classical Fourier Analysis}, Third edition, Graduate Texts in Mathematics, {\bf 249}, Springer, New York, 2014.

\bibitem{KMS} \textsc{T. Kuusi, G. Mingione, Y. Sire}, {\it Nonlocal equations with measure data,} Comm. Math. Phys. {\bf 337}, (2015), 1317-1368.

\bibitem{LPPS} {\sc T. Leonori, I. Peral, A. Primo, F. Soria}, {\em Basic estimates for solutions of a class of nonlocal elliptic and parabolic equations}. Discrete and Continuous
Dynamical Systems- A,  Volume 35, Number 12, (2015) 6031-6068.


\bibitem{Lio} \textsc{J. L. Lions},\emph{ Quelques m\'ethodes de r\'esolution des probl\'emes aux limites nonlin\'eaires } Edition Dunod, Paris 1969

\bibitem{MRT} \textsc{J.M. Mazón, J.D. Rossi, J. Toledo}, {\it Fractional p-Laplacian evolution equations.} J. Math. Pures Appl. (9) 105 (2016), no. 6, 810-844.

\bibitem{MMP}  \textsc{S. Merchán, L. Montoro, I. Peral}, \emph{Optimal reaction exponent for some qualitative properties of solutions to the p -heat equation.} Commun. Pure Appl. Anal. 14 (2015), no. 1, 245-268.

\bibitem{Ponce}\textsc{{A. C. Ponce}}, \emph{Elliptic PDEs, Measures and Capacities}, Tracts in Mathematics 23, European Mathematical Society (EMS), Zurich, 2016.

\bibitem{Pri}{\sc  A. Prignet}, {\em Existence and uniqueness of "entropy" solutions of parabolic problems with $L^1$ data}, Nonlinear Anal. 28 (1997), no. 12, 1943--1954.


\bibitem{VAZ}{\sc  J.L. Vazquez}, {\em The Dirichlet problem for the fractional p-Laplacian evolution equation.} J. Differential Equations 260 (2016), no. 7, 6038-6056.
\end{thebibliography}
\end{document}